\documentclass[10pt]{amsart}
\usepackage{amsmath,amssymb,latexsym,esint,cite,mathrsfs}
\usepackage{verbatim,wasysym}
\usepackage[left=2.6cm,right=2.7cm,top=2.9cm,bottom=2.9cm]{geometry}
\usepackage{tikz,enumitem,graphicx, subfig, microtype, color}
\usepackage{epic,eepic}

\usepackage[colorlinks=true,urlcolor=blue, citecolor=red,linkcolor=blue,
linktocpage,pdfpagelabels, bookmarksnumbered,bookmarksopen]{hyperref}
\usepackage[hyperpageref]{backref}
\usepackage[english]{babel}

\numberwithin{equation}{section}

\newtheorem{thm}{Theorem}[section]
\newtheorem{lem}[thm]{Lemma}
\newtheorem{cor}[thm]{Corollary}
\newtheorem{Prop}[thm]{Proposition}

\newtheorem{Rem}[thm]{Remark}

\newcommand{\N}{\mathbb{N}}

\newcommand{\R}{\mathbb{R}}

\newcommand\cm{\mathcal{M}}
\newcommand\cn{\mathcal{N}}

\begin{document}

\title[Critical coupled Hartree system]{High energy positive solutions for a coupled Hartree
system with Hardy-Littlewood-Sobolev critical exponents}

\author[F.\ Gao]{Fashun Gao$^\dag$}
\author[H.\ Liu]{Haidong Liu$^\ddag$}
\author[V.\ Moroz]{Vitaly Moroz$^\star$}
\author[M.\ Yang]{Minbo Yang$^{\ast\star}$}

\address{Fashun Gao,
\newline\indent Department of Mathematics and Physics, Henan University
of Urban Construction,
\newline\indent Pingdingshan, Henan, 467044, People's Republic of China}
\email{fsgao@zjnu.edu.cn}

\address{Haidong Liu,
\newline\indent Institute of Mathematics, Jiaxing University,
\newline\indent	Jiaxing, Zhejiang, 314000, People's Republic of China}
\email{liuhaidong@mail.zjxu.edu.cn}

\address{Vitaly Moroz,
\newline\indent Mathematics Department, Swansea University,
\newline\indent Bay Campus, Fabian Way, Swansea SA1 8EN, Wales, United Kingdom}
\email{v.moroz@swansea.ac.uk}

\address{Minbo Yang,
\newline\indent Department of Mathematics, Zhejiang Normal University,
\newline\indent	Jinhua, Zhejiang, 321004, People's Republic of China}
\email{mbyang@zjnu.edu.cn}

\subjclass[2010]{35A15, 35J20, 35J60}

\keywords{Coupled Hartree system; Critical nonlocal nonlinearity;
Hardy-Littlewood-Sobolev inequality.}

\thanks{$^\dag$Fashun Gao is partially supported by NSFC (11901155,
11671364).}

\thanks{$^\ddag$Haidong Liu is partially supported by NSFC (11701220,
11926334, 11926335).}

\thanks{$^\star$Vitaly Moroz and Minbo Yang are partially supported by the Royal Society IEC\textbackslash NSFC\textbackslash 191022.}

\thanks{$^\ast$Minbo Yang is partially supported by NSFC (11571317, 11971436, 12011530199) and ZJNSF(LD19A010001).}

\begin{abstract}
We study the coupled Hartree system
$$
\left\{\begin{array}{ll}
-\Delta u+ V_1(x)u
=\alpha_1\big(|x|^{-4}\ast u^{2}\big)u+\beta \big(|x|^{-4}\ast v^{2}\big)u
&\mbox{in}\ \R^N,\\[1mm]
-\Delta v+ V_2(x)v
=\alpha_2\big(|x|^{-4}\ast v^{2}\big)v +\beta\big(|x|^{-4}\ast u^{2}\big)v
&\mbox{in}\ \R^N,
\end{array}\right.
$$
where $N\geq5$,
$\beta>\max\{\alpha_1,\alpha_2\}\geq\min\{\alpha_1,\alpha_2\}>0$,
and $V_1,\,V_2\in L^{N/2}(\R^N)\cap L_{\text{loc}}^{\infty}(\R^N)$ are
nonnegative potentials. This system is critical in the sense of the Hardy-Littlewood-Sobolev inequality.
For the system with $V_1=V_2=0$ we employ moving sphere arguments in integral form to classify positive solutions and to prove the uniqueness of positive solutions up to translation and dilation, which is of independent interest.
Then using the uniqueness property, we establish a nonlocal version of the global compactness lemma and prove the existence of a high energy positive solution for the system assuming that $|V_1|_{L^{N/2}(\R^N)}+|V_2|_{L^{N/2}(\R^N)}>0$ is suitably small.
\end{abstract}

\maketitle

\begin{center}
\begin{minipage}{8.5cm}
	\small
	\tableofcontents
\end{minipage}
\end{center}
\bigskip

\section{Introduction and main results}

The two-component coupled Hartree system
\begin{equation}\label{eq1.1}
\left\{\begin{array}{ll}
i\partial_{t} \mathbf{\Psi}_1
=-\Delta \mathbf{\Psi}_1+W_1(x)\mathbf{\Psi}_1
-\alpha_1\big(K(x)\ast|\mathbf{\Psi}_1|^{2}\big)\mathbf{\Psi}_1
-\beta\big(K(x)\ast|\mathbf{\Psi}_2|^{2}\big)\mathbf{\Psi}_1,
&(t,x)\in \R^+\times \R^N,\\[1mm]
i\partial_{t} \mathbf{\Psi}_2
=-\Delta \mathbf{\Psi}_2+W_2(x)\mathbf{\Psi}_2
-\alpha_2\big(K(x)\ast|\mathbf{\Psi}_2|^{2}\big)\mathbf{\Psi}_2
-\beta\big(K(x)\ast|\mathbf{\Psi}_1|^{2}\big)\mathbf{\Psi}_2,
&(t,x)\in \R^+\times \R^N
\end{array}\right.
\end{equation}
appears in several physical models, such as in the nonlinear
optics \cite{MS} and in the study of a two-component Bose-Einstein Condensate \cite{EGBB,D}.
Here, $\mathbf{\Psi}_i:\mathbb{R}^+\times \mathbb{R}^N \to\mathbb{C}$, $W_i$ are the
external potentials, $K$ is a nonnegative response function which possesses
information about the self-interaction between the particles and $\alpha_i$ measures the strength of the self-interactions in each component: $\alpha_i>0$ corresponds to the attractive (focusing) and $\alpha_i<0$ to the repulsive (defocusing) self-interactions. The coupling constant  $\beta>0$ corresponds to the attraction (cooperation) and $\beta<0$ to the repulsion (competition) between the two components in the system.
In this work we are interested in the purely attractive case $\alpha_i,\beta>0$.
We refer the reader to \cite{LNR14,LNR16} and references therein for the physical background and mathematical derivation of Hartree theory in the case of a single equation.

When the response function is a delta function, i.e., $K(x)=\delta(x)$, the
nonlinear response is local and the problem has been intensively studied in the
past twenty years. In this case, via the ansatz
$\mathbf{\Psi}_1(t, x)=e^{-iE_1t} u(x)$ and $\mathbf{\Psi}_2(t, x)=e^{-iE_2t}v(x)$,
\eqref{eq1.1} is transformed into a coupled nonlinear Schr\"odinger system
\begin{equation}\label{eq1.2}
\left\{\begin{array}{ll}
-\Delta u+V_1(x)u=\alpha_1 u^3+\beta uv^2 &\text{in}\ \R^N,\\[1mm]
-\Delta v+V_2(x)v=\alpha_2 v^3+\beta u^2v &\text{in}\ \R^N,
\end{array} \right.
\end{equation}
where $V_i(x)=W_i(x)-E_i$ for $i=1, 2$. Existence, multiplicity and properties
of weak solutions of \eqref{eq1.2} have been investigated by many authors.
See, for example, \cite{AC, BWW, CZ1, CZ2, CLZ, DaW, LinWei1, LL,
LL1, LW, MMP, PPW, SWa, SB, TV, WY} and references therein.
Clearly, \eqref{eq1.2} may have  semitrivial solutions of the form $(u,0)$ for some $u\neq0$ or $(0,v)$ for some $v\neq0$. Looking for nontrivial solutions of \eqref{eq1.2} with both components being nonzero is more complicated and requires new techniques and ideas. Here we only recall some results closely related to the current paper.
Chen and Zou \cite{CZ1, CZ2} investigated the nonlinear Schr\"odinger system
$$
\left\{\begin{array}{ll}
-\Delta u+\lambda_1 u=\alpha_1 u^{2p-1}+\beta u^{p-1}v^{p}
&\mbox{in}\ \Omega,\\[1mm]
-\Delta v+\lambda_2 v=\alpha_2 v^{2p-1}+\beta u^{p}v^{p-1}
&\mbox{in}\ \Omega,\\[1mm]
u,\,v\geq0\ \mbox{in}\ \Omega,\ \  u=v=0\ \mbox{on}\ \partial\Omega,
&
\end{array}\right.
$$
where $\Omega\subset \R^{N}$ is a smooth bounded domain,
$2p=\frac{2N}{N-2}$ is the Sobolev critical exponent,
$-\lambda_1(\Omega)<\lambda_1,\,\lambda_2<0$, $\alpha_{1},\,\alpha_{2}>0$,
and $\beta\neq0$. They established the existence,
uniqueness and limit behaviour of positive least energy solution. It turned out
that results in the higher dimensions are quite different from those in $N=4$. In \cite{CP}, considering the functional constrained on a subset of the Nehari manifold consisting of functions invariant with respect to a subgroup of $O(N +1)$, the authors obtained infinitely many positive solutions. In \cite{SWa}, the authors showed that the Palais-Smale condition holds at any levels in a small right neighbourhood of the least energy and then proved via a contradiction argument that there is a positive solution with critical value in this small neighbourhood. Using positive solutions of the corresponding scalar equation as building blocks, the authors of \cite{PST} constructed positive solutions for systems via the Lyapunov-Schmidt reduction argument, revealing concentration and blow-up features as well as a tower shape of the solutions. Recently, Liu
and Liu \cite{LL1} considered the nonlinear Schr\"odinger system with critical
nonlinearities
\begin{equation}\label{eq1.3}
\left\{\begin{array}{l}
-\Delta u+V_1(x)u=\alpha_{1}u^{3}+\beta uv^{2}\ \ \mbox{in}\ \R^{4},\\[1mm]
-\Delta v+V_2(x)v=\alpha_{2}v^{3}+\beta u^{2}v\ \ \,\mbox{in}\ \R^{4},\\[1mm]
u,\,v\geq0\ \mbox{in}\ \R^{4},\ \ u, v\in D^{1,2}(\R^{4}),\\
\end{array}
\right.
\end{equation}
where $V_1,\,V_2\in L^2(\R^4)\cap L_{\text{loc}}^{\infty}(\R^4)$ are
nonnegative potential functions. If $\beta>\max\{\alpha_1,\alpha_2\}\ge \min\{\alpha_1,\alpha_2\}>0$ and
$|V_1|_{L^2(\R^4)}+|V_2|_{L^2(\R^4)}>0$ is suitably small, they proved that
\eqref{eq1.3} has at least a positive solution with a high energy level.
This generalizes the well known result for semilinear Schr\"odinger equation
by Benci and Cerami \cite{BC} to the coupled nonlinear Schr\"odinger system.

In this work we study standing wave solutions of \eqref{eq1.1} in the purely attractive case $\alpha_i,\beta>0$
with a Riesz potential response function, i.e.,
$K(x)= |x|^{-\mu}$ where $\mu\in(0,N)$. Clearly, $\mathbf{\Psi}_1(t,x)=e^{-iE_1t}u(x)$ and
$\mathbf{\Psi}_2(t, x)=e^{-iE_2t}v(x)$ solve \eqref{eq1.1} if and only if
$(u(x), v(x))$ is a solution of the system
\begin{equation}\label{eq1.4}
\left\{\begin{array}{ll}
-\Delta u +V_1(x)u=\alpha_1\big(|x|^{-\mu}\ast u^2\big)u
+\beta\big(|x|^{-\mu}\ast v^2\big)u
&\text{in}\ \R^N,\\[1mm]
-\Delta v +V_2(x)v=\alpha_2\big(|x|^{-\mu}\ast v^2\big)v
+\beta\big(|x|^{-\mu}\ast u^2\big)v
&\text{in}\ \R^N,
\end{array}\right.
\end{equation}
where again $V_i(x)=W_i(x)-E_i$ for $i=1, 2$. There are very few results available on
the coupled Hartree system of type \eqref{eq1.4}. The first attempt is due to Yang, Wei and Ding \cite{YWD}, there the authors considered a singular perturbed
problem related to \eqref{eq1.4} and proved the existence of a ground
state solution when the coupling constant $\beta$ is large.
In
\cite{WS}, Wang and Shi studied \eqref{eq1.4} with positive constant
potentials and proved the
existence and nonexistence of positive ground state solutions. For the critical case, the
authors of \cite{ZSSY} considered a critical coupled Hartree system with a
fractional Laplacian operator and proved the existence of a ground state
solution via the Dirichlet-to-Neumann map.

 Clearly,
semitrivial solutions of \eqref{eq1.4} correspond to solutions of the Choquard
type equation
$$
-\Delta u+ V_i(x)u=\alpha_i\big(|x|^{-\mu}\ast u^2\big)u\ \ \mbox{in}\ \R^N.
$$
The Choquard type equation goes back to the description of
the quantum theory of a polaron at rest by Pekar in 1954 \cite{Ps} and the
modelling of an electron trapped in its own hole in 1976 in the work of Choquard
\cite{L1}. They also appear in the study of boson stars \cite{ES}.
Mathematically, Lieb \cite{L1} and Lions \cite{Ls} studied the existence of solutions to the Choquard's equation.
Lieb \cite{L1} also established the uniqueness of the ground state solution when $N=3$ and $\mu=-1$.
Ma and Zhao \cite{MZ} studied symmetry and uniqueness of positive solutions.
The existence of radial ground states with nonlinearities more general than $u^2$  was studied in \cite{MS1,MS3}.
The authors of \cite{GSYZ} considered the Choquard equation
\begin{equation}\label{eq1.6}
-\Delta u+V(x)u=\big(|x|^{-\mu}\ast |u|^{2_{\mu}^{\ast}}\big)|u|^{2_{\mu}^{\ast}-2}u
\ \ \mbox{in}\ \R^N,
\end{equation}
where $2_{\mu}^{\ast}=\frac{2N-\mu}{N-2}$ is the critical exponent.
They established a global compactness result and proved that \eqref{eq1.6}
has at least a positive solution if $|V|_{L^{N/2}(\R^N)}>0$ is suitable small.
This extended to the nonlocal Choquard equation the well known result for semilinear Schr\"odinger equation by Benci
and Cerami \cite{BC}.
The existence of multiple solutions for \eqref{eq1.6} was established in \cite{AFM}.
Lei \cite{Lei}, Du and Yang \cite{DY} studied positive solutions of the critical equation
\begin{equation}\label{eq1.5}
-\Delta u=\big(|x|^{-\mu}\ast |u|^{2_{\mu}^{\ast}}\big)|u|^{2_{\mu}^{\ast}-2}u
\ \ \mbox{in}\ \R^N,
\end{equation}
and proved that every positive solution of \eqref{eq1.5} must assume the form
$$
u(x)=c\Big(\frac{\delta}{\delta^{2}+|x-z|^{2}}\Big)^{\frac{N-2}{2}}.
$$
They also established the nondegeneracy result when $\mu$ is close to
$N$. We also refer the readers to \cite{Acker, AGSY, ANY, GSYZ, MZ, WW} and a survey \cite{MS7} for recent progress on the topic of Choquard equation.

Inspired by the work in \cite{BC, LL} for local nonlinear Schr\"odinger equations and \cite{CeMo} for a Schr\"odinger--Poisson system,
in the present paper we aim to study the existence of nontrivial solutions of the critical Hartree system
\begin{equation}\label{eq1.7}
\left\{\begin{array}{ll}
-\Delta u+ V_1(x)u
=\alpha_1\big(|x|^{-4}\ast u^{2}\big)u+\beta \big(|x|^{-4}\ast v^{2}\big)u
&\mbox{in}\ \R^N,\\[1mm]
-\Delta v+ V_2(x)v
= \alpha_2\big(|x|^{-4}\ast v^{2}\big)v +\beta\big(|x|^{-4}\ast u^{2}\big)v
&\mbox{in}\ \R^N,
\end{array}\right.
\end{equation}
where $N\geq5$,
$\beta>\max\{\alpha_1,\alpha_2\}\geq\min\{\alpha_1,\alpha_2\}>0$,
and $V_1,\,V_2\in L^{\frac{N}{2}}(\R^N)\cap L_{\text{loc}}^{\infty}(\R^N)$
are nonnegative potential functions. Our goal is to find a nontrivial positive solution of \eqref{eq1.7} at a higher energy level when $L^2$-norm of $V_i$ are both suitably small.

To state the main results, we first recall
the Hardy-Littlewood-Sobolev inequality (see \cite[Theorem 4.3]{LLo}) to clarify the meaning of "critical" for the nonlocal Hartree equation.

\begin{Prop}\label{pro1.1}
Let $t,\,r>1$ and $0<\mu<N$ be such that $\frac{1}{t}+\frac{\mu}{N}+\frac{1}{r}=2$.
Then there is a sharp constant $C(N,\mu,t)$ such that, for $f\in L^{t}(\R^N)$
and $h\in L^{r}(\R^N)$,
\begin{equation}\label{eq1.8}
\left|\int_{\R^{N}}\int_{\R^{N}}\frac{f(x)h(y)}{|x-y|^{\mu}}dxdy\right|
\leq C(N,\mu,t) |f|_{L^t(\R^N)}|h|_{L^r(\R^N)}.
\end{equation}
In particular, if $t=r=\frac{2N}{2N-\mu}$ then
$$
C(N,\mu,t)=C(N,\mu)=\pi^{\frac{\mu}{2}}
\frac{\Gamma(\frac{N}{2}-\frac{\mu}{2})}{\Gamma(N-\frac{\mu}{2})}
\bigg(\frac{\Gamma(\frac{N}{2})}{\Gamma(N)}\bigg)^{-1+\frac{\mu}{N}},
$$
where $\Gamma(s)=\int_0^{+\infty} x^{s-1}e^{-x}\,dx$, $s>0$. In this
case, equality in \eqref{eq1.8} holds if and only if $f\equiv(const.)\,h$
and
$$
h(x)=c\big(\delta^{2}+|x-z|^{2}\big)^{-\frac{2N-\mu}{2}}
$$
for some $c\in \mathbb{R}$, $\delta>0$ and $z\in \R^{N}$.
\end{Prop}

According to Proposition \ref{pro1.1}, the functional
$$
\int_{\R^{N}}\int_{\R^{N}}\frac{|u(x)|^{p}|v(y)|^{p}}{|x-y|^{\mu}}dxdy
$$
is well defined in $H^1(\R^N)\times H^1(\R^N)$ if
$\frac{2N-\mu}{N}\leq p\leq\frac{2N-\mu}{N-2}$.
Here the constant $2_{\mu}^{\ast}=\frac{2N-\mu}{N-2}$ is called the upper
Hardy-Littlewood-Sobolev critical exponent.
In this sense, \eqref{eq1.7} is said to be a critical Hartree system.

The authors of \cite{GY} investigated a critical Choquard type equation on a bounded domain and extended the well known
results in \cite{BN}. In particular, it was proved in \cite{GY} that the infimum
$$
S_{H,L}:=
\inf\limits_{u\in D^{1,2}(\R^N)\backslash\{{0}\}}\ \
\frac{\displaystyle\int_{\R^N}|\nabla u|^{2}dx}{\displaystyle\big(\int_{\R^N}\int_{\R^N}
\frac{u^2(x)u^2(y)}{|x-y|^4}dxdy\big)^{\frac12}}
$$
is achieved if and only if
$$
u(x)=c\Big(\frac{\delta}{\delta^{2}+|x-z|^{2}}\Big)^{\frac{N-2}{2}},
$$
where $c>0$, $\delta>0$ and $z\in \R^{N}$. Moreover,
\begin{equation}\label{eq1.9}
S_{H,L}=\frac{S}{\sqrt{C(N,4)}},
\end{equation}
where $S$ is the optimal constant for the Sobolev embedding
$D^{1,2}(\R^N)\hookrightarrow L^{\frac{2N}{N-2}}(\R^N)$.

Our first step in this work is to establish a classification of positive solutions for the critical Hartree system
\begin{equation}\label{eq3.1+}
\left\{\begin{array}{ll}
-\Delta u=\alpha_1\big(|x|^{-4}\ast u^{2}\big)u+\beta \big(|x|^{-4}\ast v^{2}\big)u
&\mbox{in}\ \R^N,\\[1mm]
-\Delta v=\alpha_2\big(|x|^{-4}\ast v^{2}\big)v +\beta\big(|x|^{-4}\ast u^{2}\big)v
&\mbox{in}\ \R^N,
\end{array}\right.
\end{equation}
which plays a role of the limit system for \eqref{eq1.7}.
To formulate our result denote $k_{0}=\frac{\beta-\alpha_2}{\beta^{2}-\alpha_1\alpha_2}$,  $l_{0}=\frac{\beta-\alpha_1}{\beta^{2}-\alpha_1\alpha_2}$ and $R_{N}=\frac14 \pi^{-\frac{N}{2}}\Gamma(\frac{N-2}{2})$. For
$0<s<\frac{N}{2}$ we set
$$
I(s)=\frac{\pi^{\frac{N}{2}}\Gamma(\frac{N-2s}{2})}{\Gamma(N-s)}.
$$
Using moving sphere arguments in integral form inspired by \cite{CLO, DY, Lei} we establish the uniqueness of positive solutions of  \eqref{eq3.1+} up to translation and dilation.

\begin{thm}\label{thm3.1+}
	Let $\beta>\max\{\alpha_1,\alpha_2\}\geq\min\{\alpha_1,\alpha_2\}>0$. If $(u, v)\in H:=D^{1,2}(\mathbb{R}^{N})\times D^{1,2}(\mathbb{R}^{N})$ is a positive classical solution of \eqref{eq3.1+}, then
	$$
	u(x)=C_{1}\Big(\frac{\tau}{\tau^{2}+|x-\overline{x}|^{2}}\Big)^{\frac{N-2}{2}},
	\ \ v(x)=C_{2}\Big(\frac{\tau}{\tau^{2}+|x-\overline{x}|^{2}}\Big)^{\frac{N-2}{2}}
	$$
	for some $\tau>0$ and $\overline{x}\in\R^N$, where
	$$
	C_{1}=\frac{\sqrt{k_{0}}}{\sqrt{R_{N}I(2)I\big(\frac{N-2}{2}\big)}},
	\ \ C_{2}=\frac{\sqrt{l_{0}}}{\sqrt{R_{N}I(2)I\big(\frac{N-2}{2}\big)}}.
	$$
\end{thm}

To study \eqref{eq1.7} using variational methods we introduce the energy functional
$$
\aligned
\mathcal J(u,v)&=
\frac{1}{2}\int_{\R^N}\big(|\nabla u|^{2}+|\nabla v|^{2}+V_1(x)u^{2}+ V_2(x)v^{2}\big)dx\\
&\quad\,-\frac{1}{4}\int_{\R^N}\int_{\R^N}
\frac{\alpha_1|u(x)|^{2}|u(y)|^{2}+\alpha_2|v(x)|^{2}|v(y)|^{2}
	+2\beta|u(x)|^{2}|v(y)|^{2}}{|x-y|^{4}}dxdy,
\endaligned
$$
In view of the Hardy-Littlewood-Sobolev inequality, the functional $\mathcal J$ is of class $\mathcal{C}^{1}$ on the Hilbert space $H:=D^{1,2}(\mathbb{R}^{N})\times D^{1,2}(\mathbb{R}^{N})$. Critical points of $\mathcal J$ are weak solutions of \eqref{eq1.7}.
Consider the infimum
\begin{equation}\label{e-inf}
c = \inf_{(u,v)\in\tilde{\mathcal{N}}}\mathcal J(u, v),
\end{equation}
where $\tilde{\mathcal{N}} = \{(u, v)\in H : (u, v) \neq (0, 0),\,\langle \mathcal J'(u, v),(u, v)\rangle=0\}$
is the Nehari manifold of $\mathcal J$.
A ground state solution of \eqref{eq1.7} is {\it by definition} a minimizer of \eqref{e-inf}.

In the case $V_1=V_2=0$, the analogues of $\mathcal J$ and $c$ are
denoted by $\mathcal J_{\infty}$ and $c_{\infty}$ respectively. We prove in Lemma \ref{lem2.3}
that if $\beta>\max\{\alpha_1,\alpha_2\}$, then
$$
c_{\infty}=\frac{1}{4}(k_{0}+l_{0})S_{H,L}^{2}.
$$
This estimate, combined with the uniqueness result of Theorem \ref{thm3.1+} implies that
{\em every finite energy positive classical solution of the limit system \eqref{eq3.1+} is a ground state solution} (see Corollary \ref{cor3.2}).

We also prove in Lemma \ref{lem2.4}
that if $V_j\neq0$ for some $j\in\{1,2\}$ and $\beta>\max\{\alpha_1,\alpha_2\}$ then
$$c=c_\infty$$
and the infimum in \eqref{e-inf} is not attained, i.e.~\eqref{eq1.7} does not have ground state solutions.

The main result of this paper is the following.

\begin{thm}\label{thm1.2}
Let $N\geq5$ and $\beta>\max\{\alpha_1,\alpha_2\}\geq\min\{\alpha_1,\alpha_2\}>0$.
If $V_1,\,V_2\in L^{\frac{N}{2}}(\R^N)\,\cap\, L_{\text{\rm loc}}^{\infty}(\R^N)$
are nonnegative functions satisfying
\begin{equation}\label{eq1.10}
\aligned
0
&<\frac{\beta-\alpha_2}{2\beta-\alpha_1-\alpha_2}C(N, 4)^{-\frac12} |V_1|_{L^{N/2}(\R^N)}
+\frac{\beta-\alpha_1}{2\beta-\alpha_1-\alpha_2}C(N, 4)^{-\frac12} |V_2|_{L^{N/2}(\R^N)}\\
&\hspace{4mm}<\min\bigg\{\sqrt{\frac{\beta^{2}-\alpha_1\alpha_2}{\alpha_1(2\beta-\alpha_1-\alpha_2)}},
\sqrt{\frac{\beta^{2}-\alpha_1\alpha_2}{\alpha_2(2\beta-\alpha_1-\alpha_2)}},\sqrt{2}\bigg\}
S_{H,L}-S_{H,L},
\endaligned
\end{equation}
then system \eqref{eq1.7} has a positive solution $(u,v)\in H$ such that
$$c_\infty=c<\mathcal J(u,v)\le\min\Big\{\frac{S_{H,L}^{2}}{4\alpha_1}, \frac{S_{H,L}^{2}}{4\alpha_2}, 2c_{\infty}\Big\}.$$
\end{thm}

The main difficulties in the proof of Theorem \ref{thm1.2} are the loss of compactness and the need to distinguish the solutions from the semitrivial ones. To overcome the loss of compactness we will follow the idea of Struwe \cite{Sm} to establish a novel nonlocal version of the global compactness lemma for the critical Hartree system (Lemma \ref{lem4.2}). This result is more delicate than the lemma obtained in \cite{GSYZ} for the scalar equation, since the proofs in this work rely heavily on the classification of positive solutions of the limit system established in Theorem \ref{thm3.1+}.

\begin{Rem}
{\rm The proofs in this paper can also be adapted to the general form of the critical Hartree system
\begin{equation}\label{equation2}
\left\{\begin{array}{ll}
-\Delta u+ V_1(x)u
=\alpha_1\big(|x|^{-\mu}\ast |u|^{2_{\mu}^{\ast}}\big)|u|^{2_{\mu}^{\ast}-2}u+\beta \big(|x|^{-\mu}\ast |v|^{2_{\mu}^{\ast}}\big)|u|^{2_{\mu}^{\ast}-2}u
&\mbox{in}\ \R^N,\\[1mm]
-\Delta v+ V_2(x)v
= \alpha_2\big(|x|^{-\mu}\ast |v|^{2_{\mu}^{\ast}}\big)|v|^{2_{\mu}^{\ast}-2}v +\beta\big(|x|^{-\mu}\ast u^{2_{\mu}^{\ast}}\big)|v|^{2_{\mu}^{\ast}-2}v
&\mbox{in}\ \R^N,
\end{array}\right.
\end{equation}
where $N\geq3$, $N>\mu>0$,
$\beta>\frac{N-\mu+2}{N-2}\max\{\alpha_1,\alpha_2\}\geq\min\{\alpha_1,\alpha_2\}>0$,
and $V_1,\,V_2\in L^{\frac{N}{2}}(\R^N)\cap L_{\text{loc}}^{\infty}(\R^N)$
are nonnegative potential functions with small $L^{\frac{N}{2}}$ norms.}
\end{Rem}

Throughout this paper, we will use the following notations.
\begin{itemize}
\item The standard norm in the Sobolev space $D^{1,2}(\mathbb{R}^{N})$
is given by
$$\|u\|:=\Big(\int_{\mathbb{R}^N}|\nabla u|^{2}dx\Big)^{\frac12}.$$

\item Set $H:=D^{1,2}(\mathbb{R}^{N})\times D^{1,2}(\mathbb{R}^{N})$
equipped with the norm
$$\|(u,v)\|:=\Big(\int_{\mathbb{R}^N}(|\nabla u|^{2}+|\nabla v|^{2})dx\Big)^{\frac12}.$$

\item The standard norm in $L^{q}(\Omega)$ is denoted by $|\cdot|_{q,\Omega}$
and by $|\cdot|_{q}$ if $\Omega=\R^N$.

\item $o(1)$ means a quantity which tends to 0.

\item $c, C_j, C_N, C(\cdot)$ stand for various positive constants whose exact
values are irrelevant.
\end{itemize}

The paper is organized as follows. In Section 2, we give some preliminary
results. In Section 3, we prove a uniqueness result for limit system by the
method of moving spheres. Section 4 is devoted to the proof of a nonlocal
global compactness lemma and Theorem \ref{thm1.2} is proved in Section 5.

\section{Preliminaries}
To prove the existence of positive solutions of \eqref{eq1.7}, we will study the
modified system
\begin{equation}\label{eq2.1}
\left\{\begin{array}{ll}
\displaystyle-\Delta u+ V_1(x)u
=\alpha_1\big(|x|^{-4}\ast |u^{+}|^{2}\big)u^{+}+\beta\big(|x|^{-4}\ast |v^{+}|^{2}\big)u^{+}
&\mbox{in}\ \R^N,\\[1mm]
\displaystyle -\Delta v+ V_2(x)v
= \alpha_2\big(|x|^{-4}\ast |v^{+}|^{2}\big)v^{+}+\beta\big(|x|^{-4}\ast |u^{+}|^{2}\big)v^{+}
&\mbox{in}\ \R^N,
\end{array}\right.
\end{equation}
where $u^{+}=\max\{u,0\}$ and $v^{+}=\max\{v,0\}$. In fact, if $(u, v)$ is
a nontrivial solution of \eqref{eq2.1}, then $u(x)>0$ and $v(x)>0$ for all $x\in\mathbb R^N$ by the strong maximum principle,
which implies that $(u, v)$ is a positive solution of \eqref{eq1.7}. Therefore, we only need to prove the
existence of a nontrivial solution of \eqref{eq2.1}.

Semitrivial solutions of \eqref{eq2.1} are closely related to solutions of the
single elliptic equation
$$
-\Delta u+V_{i}(x)u=\alpha_{i}\big(|x|^{-4}\ast |u^{+}|^{2}\big)u^{+}\ \
\mbox{in}\ \R^N,
$$
of which the associated functional $I_i : D^{1,2}(\mathbb R^N)\to\mathbb R$
is defined by
$$
I_{i}(u)=\frac{1}{2}\int_{\R^N}\big(|\nabla u|^{2}+ V_{i}(x)u^{2}\big)dx
-\frac{\alpha_{i}}{4}\int_{\R^N}\int_{\R^N}
\frac{|u^{+}(x)|^{2}|u^{+}(y)|^{2}}{|x-y|^{4}}dxdy.
$$
Consider the infimum
$$
c_{i} = \inf_{u\in\mathcal{M}_{i}}I_{i}(u),
$$
where
$\mathcal{M}_{i} = \{u\in D^{1,2}(\R^{N}) : u \neq 0,\,\langle I_{i}'(u), u\rangle=0\}$.
In the case $V_i= 0$, we denote the analogues of $I_{i}$, $c_{i}$, $\mathcal{M}_{i}$
by $I_{i\infty}$, $c_{i\infty}$, $\mathcal{M}_{i\infty}$ respectively.

For $\delta>0$ and $z\in\R^N$, we denote
$$
U_{\delta,z}(x)=C_{N}\Big(\frac{\delta}{\delta^{2}+|x-z|^{2}}\Big)^{\frac{N-2}{2}}
$$
where $C_N=S^{-\frac{N-4}{4}}C(N, 4)^{-\frac12}[N(N-2)]^{\frac{N-2}{4}}$.
Then we have
$$
\int_{\R^N}|\nabla U_{\delta,z}|^{2}dx
=\int_{\R^N}\int_{\R^N}\frac{U^{2}_{\delta,z}(x) U^{2}_{\delta,z}(y)}{|x-y|^{4}}dxdy
=S_{H,L}^{2}
$$
and, according to \cite[Theorem 1.3]{DY}, the set
$\{U_{\delta,z} : \delta>0, z\in\R^N\}$ contains all positive solutions of
$$
-\Delta u=\big(|x|^{-4}\ast u^{2}\big)u\ \ \hbox{in}\ \R^N.
$$
It is easy to verify that $c_{i\infty}=\frac{1}{4\alpha_{i}}S_{H,L}^{2}$.

\begin{lem}\label{lem2.1}
If $V_i\in L^{\frac{N}{2}}(\R^N)$ is nonnegative, then
$c_{i}=c_{i\infty}=\frac{1}{4\alpha_{i}}S_{H,L}^{2}$.
\end{lem}

\begin{proof}[\bf Proof.]
For $u \in \mathcal{M}_{i}$, let $t_u > 0$ be such that
$t_uu \in\mathcal{M}_{i\infty}$. Since $V_i(x)\geq0$ for $x\in\R^N$, we
have
$$
t_u^{2}
=\frac{\|u\|^{2}}{\displaystyle\alpha_{i}\int_{\R^N}\int_{\R^N}
\frac{|u^{+}(x)|^{2}|u^{+}(y)|^{2}}{|x-y|^{4}}dxdy}
\leq\frac{\displaystyle\|u\|^{2}+\int_{\R^N}V_{i}(x)u^{2}dx}{\displaystyle\alpha_{i}\int_{\R^N}\int_{\R^N}
\frac{|u^{+}(x)|^{2}|u^{+}(y)|^{2}}{|x-y|^{4}}dxdy}
=1.
$$
Then
$$
c_{i\infty}\leq I_{i\infty}(t_uu)=\frac{1}{4} t_u^{2}\|u\|^{2}
\leq\frac{1}{4}\Big(\|u\|^{2}+\int_{\R^N} V_{i}(x)u^{2}dx\Big)
=I_{i}(u),
$$
which implies that $c_{i}\geq c_{i\infty}$.

To prove the reverse inequality, let $t_n > 0$ be given by
$$
t_n^{2}
=\frac{\displaystyle\|U_{1,z_{n}}\|^{2}+\int_{\R^N}V_{i}(x)U_{1,z_{n}}^{2}dx}
{\displaystyle\alpha_{i}\int_{\R^N}\int_{\R^N}
\frac{U_{1,z_{n}}^{2}(x) U_{1,z_{n}}^{2}(y)}{|x-y|^{4}}dxdy},
$$
then we have $t_n U_{1,z_{n}}\in\mathcal{M}_{i}$. Therefore
$$
\aligned
c_{i}\leq I_{i}(t_nU_{1,z_{n}})
&=\frac{1}{4} t_n^{2}\Big(\|U_{1,z_{n}}\|^{2}
+\int_{\R^N} V_{i}(x)U_{1,z_{n}}^{2}dx\Big).
\endaligned
$$
Since $V_i\in L^{\frac{N}{2}}(\R^N)$, for any $\varepsilon>0$
there exists a number $r=r(\varepsilon)>0$ such that
$$
\Big(\int_{\R^N\backslash B_r(0)}|V_i|^{\frac{N}{2}}dx\Big)^{\frac{2}{N}}
<\varepsilon.
$$
For such $r$, let $\{z_n\}\subset\R^{N}$ be a sequence such that
$\lim_{n\to\infty}|z_n|=+\infty$, then we can find
$n_{0}\in\mathbb N$ such that
$$
\Big(\int_{B_r(0)}U_{1,z_{n}}^{\frac{2N}{N-2}}dx\Big)^{\frac{N-2}{N}}
=\Big(\int_{B_r(-z_{n})}U_{1,0}^{\frac{2N}{N-2}}dx\Big)^{\frac{N-2}{N}}
<\varepsilon,\ \ \text{for}\ n\geq n_{0}.
$$
Then the H\"older inequality leads to
$$
\aligned
\Big|\int_{\R^N}V_i(x)U_{1,z_{n}}^{2}dx\Big|
&=\Big|\int_{B_r(0)}V_i(x)U_{1,z_{n}}^{2}dx
+\int_{\R^N\backslash B_r(0)}V_i(x)U_{1,z_{n}}^{2}dx\Big|\\
&\leq|V_i|_{\frac{N}{2}}\Big(\int_{B_r(0)}U_{1,z_{n}}^{\frac{2N}{N-2}}dx\Big)^{\frac{N-2}{N}}
+\Big(\int_{\R^N\backslash B_r(0)}|V_i|^{\frac{N}{2}}dx\Big)^{\frac{2}{N}}
\Big(\int_{\R^{N}}U_{1,0}^{\frac{2N}{N-2}}dx\Big)^{\frac{N-2}{N}}\\
&<C\varepsilon
\endaligned
$$
for $n\geq n_{0}$, which implies that
\begin{equation}\label{eq2.2}
\lim_{n\rightarrow\infty}\int_{\R^N}V_i(x)U_{1,z_{n}}^{2}dx=0.
\end{equation}
Moreover, by the definition of $t_n$, we know
$t_n^{2}=\frac{1}{\alpha_{i}}+o_{n}(1)$ as $n\rightarrow\infty$. Then
$$
\aligned
c_{i}\leq I_{i}(t_nU_{1,z_{n}})
=\frac{1}{4}\Big(\frac{1}{\alpha_{i}}+o_{n}(1)\Big)\big(S_{H,L}^{2}+o_{n}(1)\big)
=\frac{1}{4\alpha_{i}}S_{H,L}^{2}+o_{n}(1)
\endaligned
$$
and so $c_{i}\leq\frac{1}{4\alpha_{i}}S_{H,L}^{2}=c_{i\infty}$. Combining
this with $c_{i}\geq c_{i\infty}$ , we get the desired conclusion.
\end{proof}

To study \eqref{eq2.1} by variational methods, we define the energy
functional $J: H\to\mathbb R$ by
$$
\aligned
J(u,v)&=
\frac{1}{2}\int_{\R^N}\big(|\nabla u|^{2}+|\nabla v|^{2}+V_1(x)u^{2}+ V_2(x)v^{2}\big)dx\\
&\quad\,-\frac{1}{4}\int_{\R^N}\int_{\R^N}
\frac{\alpha_1|u^{+}(x)|^{2}|u^{+}(y)|^{2}+\alpha_2|v^{+}(x)|^{2}|v^{+}(y)|^{2}
+2\beta|u^{+}(x)|^{2}|v^{+}(y)|^{2}}{|x-y|^{4}}dxdy.
\endaligned
$$
In view of the Hardy-Littlewood-Sobolev inequality, the functional $J$ is well
defined and belongs to $\mathcal{C}^{1}(H, \mathbb R)$. Then we see that
$(u,v)$ is a weak solution of \eqref{eq2.1} if and only if $(u,v)$ is a critical
point of the functional $J$.

Consider the infimum
$$
c = \inf_{(u,v)\in\mathcal{N}}J(u, v),
$$
where
$$\mathcal{N} = \{(u, v)\in H : (u, v) \neq (0, 0),\,\langle J'(u, v),(u, v)\rangle=0\}.$$
In the case $V_1=V_2=0$, the analogues of $J$, $c$, $\mathcal{N}$ will be
denoted by $J_{\infty}$, $c_{\infty}$, $\mathcal{N}_{\infty}$ respectively. As
is known, critical points of the functional $J_{\infty}$ correspond to weak
solutions of the coupled system
\begin{equation}\label{eq2.3}
\left\{\begin{array}{ll}
-\Delta u
=\alpha_1\big(|x|^{-4}\ast |u^{+}|^{2}\big)u^{+}+\beta\big(|x|^{-4}\ast |v^{+}|^{2}\big)u^{+}
&\mbox{in}\ \R^N,\\[1mm]
-\Delta v
=\alpha_2\big(|x|^{-4}\ast |v^{+}|^{2}\big)v^{+}+\beta\big(|x|^{-4}\ast |u^{+}|^{2}\big)v^{+}
&\mbox{in}\ \R^N.
\end{array}
\right.
\end{equation}

The next lemma is proved in \cite{LL1}.

\begin{lem}\label{lem2.2}
If $\beta>\max\{\alpha_1,\alpha_2\}$ and $f : [0,+\infty)\rightarrow\R$ is
defined by
$$
f(t)=\frac{(t+1)^{2}}{\alpha_1t^{2}+2\beta t+\alpha_2},
$$
then $\min_{t\geq0}f(t)=k_{0}+l_{0}$, where
$k_{0}=\frac{\beta-\alpha_2}{\beta^{2}-\alpha_1\alpha_2}$
and $l_{0}=\frac{\beta-\alpha_1}{\beta^{2}-\alpha_1\alpha_2}$.
\end{lem}

\begin{lem}\label{lem2.3}
If $\beta>\max\{\alpha_1,\alpha_2\}$, then we have
$$
c_{\infty}=\frac{1}{4}(k_{0}+l_{0})S_{H,L}^{2}
$$
and any least energy solution of \eqref{eq2.3} must be of the form
$$
(\sqrt{k_{0}}U_{\delta,z},\sqrt{l_{0}}U_{\delta,z})
$$
for some $\delta> 0$ and $z\in \R^{N}$, where again
$k_{0}=\frac{\beta-\alpha_2}{\beta^{2}-\alpha_1\alpha_2}$
and $l_{0}=\frac{\beta-\alpha_1}{\beta^{2}-\alpha_1\alpha_2}$.
\end{lem}

\begin{proof}[\bf Proof.]
Firstly, we show that $c_{\infty}=\frac{1}{4}(k_{0}+l_{0})S_{H,L}^{2}$. Since
$(\sqrt{k_{0}}U_{\delta,z}, \sqrt{l_{0}}U_{\delta,z})\in\mathcal{N}_{\infty}$,
we have
$$
c_{\infty}
\leq J_{\infty}(\sqrt{k_{0}}U_{\delta,z}, \sqrt{l_{0}}U_{\delta,z})
=\frac{1}{4}(k_{0}+l_{0})S_{H,L}^{2}.
$$
To prove the reverse inequality, let $(u, v)\in\mathcal{N}_{\infty}$ and
assume without loss of generality that $v\neq 0$. Set
$$
t=\frac{|u|_{2^{\ast}}^{2}}{|v|_{2^{\ast}}^{2}}\geq 0,
$$
where $2^{\ast}=\frac{2N}{N-2}$. Then, by the Sobolev inequality and
Proposition \ref{pro1.1},
$$
\aligned
S(t+1)|v|_{2^{\ast}}^{2}
&=S|u|_{2^{\ast}}^{2}+S|v|_{2^{\ast}}^{2}\\
&\leq\int_{\mathbb R^N}\big(|\nabla u|^{2}+|\nabla v|^{2}\big)dx\\
&=\int_{\R^N}\int_{\R^N}
\frac{\alpha_1|u^{+}(x)|^{2}|u^{+}(y)|^{2}+\alpha_2|v^{+}(x)|^{2}|v^{+}(y)|^{2}
+2\beta|u^{+}(x)|^{2}|v^{+}(y)|^{2}}{|x-y|^{4}}dxdy\\
&\leq C(N,4)\big(\alpha_1|u|_{2^{\ast}}^{4}
+\alpha_2|v|_{2^{\ast}}^{4}+2\beta |u|_{2^{\ast}}^{2}|v|_{2^{\ast}}^{2}\big)\\
&=C(N,4)\big(\alpha_1 t^{2}+2\beta t+\alpha_2\big)|v|_{2^{\ast}}^{4}
\endaligned
$$
and so
$$
|v|_{2^{\ast}}^{2}\geq\frac{S(t+1)}{C(N,4)\big(\alpha_1t^{2}+2\beta t+\alpha_2\big)}.
$$
Noting that $(u, v)\in\mathcal{N}_{\infty}$, by Lemma \ref{lem2.2} and
\eqref{eq1.9}, we obtain
\begin{equation}\label{eq2.4}
\aligned
J_{\infty}(u,v)
&=\frac14 \int_{\mathbb R^N}\big(|\nabla u|^{2}+|\nabla v|^{2}\big)dx\\
&\geq\frac{1}{4}S(t+1)|v|_{2^{\ast}}^{2}\\
&\geq\frac{S^{2}(t+1)^{2}}{4C(N,4)\big(\alpha_1t^{2}+2\beta t+\alpha_2\big)}\\
&\geq\frac{1}{4}(k_{0}+l_{0})S_{H,L}^{2}.
\endaligned
\end{equation}
Then $c_{\infty}\geq\frac{1}{4}(k_{0}+l_{0})S_{H,L}^{2}$ and we conclude
that $c_{\infty}=\frac{1}{4}(k_{0}+l_{0})S_{H,L}^{2}$.

Secondly, we prove the uniqueness of least energy solutions up to translation
and dilation. On one hand, if $(u,v)\in\mathcal{N}_{\infty}$ and either
$u\neq a_{1}U_{\delta,z}$ or $v\neq a_{2}U_{\delta,z}$, where $a_{1}\neq 0$,
$a_{2}\neq 0$, $\delta>0$ and $z\in \R^{N}$ are parameters, then we
see from \eqref{eq2.4} and
$$
c_{i\infty}=\frac{1}{4\alpha_{i}}S_{H,L}^{2}>\frac{1}{4}(k_{0}+l_{0})S_{H,L}^{2},\ \ i=1, 2
$$
that $J_{\infty}(u,v)>\frac{1}{4}(k_{0}+l_{0})S_{H,L}^{2}$. On the other hand,
if $(u, v)=(a_{1}U_{\delta,z},a_{2}U_{\delta,z})\in\mathcal{N}_{\infty}$ and
$J_{\infty}(u,v)=c_{\infty}$, then there must be $a_{1}=\sqrt{k_{0}}$ and
$a_{2}=\sqrt{l_{0}}$. Therefore, any least energy solution of \eqref{eq2.3}
must be of the form
$$
(\sqrt{k_{0}}U_{\delta,z},\sqrt{l_{0}}U_{\delta,z})
$$
for some $\delta> 0$ and $z\in \R^{N}$. The proof is complete.
\end{proof}

\begin{lem}\label{lem2.4}
If $\beta>\max\{\alpha_1,\alpha_2\}$ and
$V_1,\,V_2\in L^{\frac{N}{2}}(\R^{N})$ are nonnegative functions satisfying
\begin{equation}\label{eq2.5}
|V_1|_{\frac{N}{2}}+|V_2|_{\frac{N}{2}}>0,
\end{equation}
then $c =c_{\infty}$ and $c$ is not achieved.
\end{lem}

\begin{proof}[\bf Proof.]
Using the arguments in the proof of Lemma \ref{lem2.1}, one can prove
$c\geq c_{\infty}=\frac{1}{4}(k_{0}+l_{0})S_{H,L}^{2}$. We shall show that
the equality holds indeed. Let us consider the sequence
$\{z_{n}\}\subset\R^{N}$ satisfying $|z_{n}|\rightarrow+\infty$ as
$n\rightarrow\infty$. By the claim in the proof of Lemma \ref{lem2.1}, we have
$$
\lim_{n\rightarrow\infty}\int_{\R^N}
\big(k_{0} V_1(x)U_{1,z_{n}}^{2}+l_{0} V_2(x)U_{1,z_{n}}^{2}\big)dx=0.
$$
For $t_n > 0$ defined by
$$
t_n^{2}=\frac{\displaystyle\|(\sqrt{k_{0}}U_{1,z_{n}},\sqrt{l_{0}}U_{1,z_{n}})\|^{2}
+\int_{\R^N}\big(k_{0} V_1(x)U_{1,z_{n}}^{2}+l_{0} V_2(x)U_{1,z_{n}}^{2}\big)dx}
{\displaystyle(k_{0}+l_{0})\int_{\R^N}\int_{\R^N}
\frac{U^{2}_{1,z_{n}}(x) U^{2}_{1,z_{n}}(y)}{|x-y|^{4}}dxdy},
$$
we have $(t_n\sqrt{k_{0}}U_{1,z_{n}},t_n\sqrt{l_{0}}U_{1,z_{n}})\in\mathcal{N}$
and $t_n^{2}=1+o_{n}(1)$ as $n\rightarrow\infty$. Then
$$
\aligned
c
&\leq J(t_n\sqrt{k_{0}}U_{1,z_{n}},t_n\sqrt{l_{0}}U_{1,z_{n}})\\
&=\frac{1}{4}t_n^{2}
\Big(\|(\sqrt{k_{0}}U_{1,z_{n}},\sqrt{l_{0}}U_{1,z_{n}})\|^{2}
+\int_{\R^N}\big(k_{0} V_1(x)U_{1,z_{n}}^{2}+l_{0} V_2(x)U_{1,z_{n}}^{2}\big)dx\Big)\\
&=\frac{1}{4}\big(1+o_{n}(1)\big)\big((k_{0}+l_{0})S_{H,L}^{2}+o_{n}(1)\big)\\
&=\frac{1}{4}(k_{0}+l_{0})S_{H,L}^{2}+o_{n}(1),
\endaligned
$$
which implies $c\leq\frac{1}{4}(k_{0}+l_{0})S_{H,L}^{2}=c_{\infty}$.
Therefore, we conclude that $c =c_{\infty}$.

We use an argument of contradiction to prove the nonexistence result.
Assume that $(u,v)\in\mathcal{N}$ satisfies
$$
J(u,v)=\frac{1}{4}
\Big(\|(u,v)\|^{2}+\int_{\R^N}\big(V_1(x)u^{2}+V_2(x)v^{2}\big)dx\Big)
=c.
$$
Let $t_{(u,v)}>0$ be such that $(t_{(u,v)}u, t_{(u,v)}v)\in\mathcal{N}_{\infty}$.
It is easy to verify that $t_{(u,v)}\leq 1$. Then
$$
c_{\infty}
\leq J_{\infty}(t_{(u,v)}u, t_{(u,v)}v)
=\frac{1}{4}t_{(u,v)}^{2}\|(u,v)\|^{2}\leq\frac{1}{4}
\Big(\|(u,v)\|^{2}+\int_{\R^N}\big(V_1(x)u^{2}+V_2(x)v^{2}\big)dx\Big)
=c=c_{\infty},
$$
which indicates that $t_{(u,v)}=1$ and
\begin{equation}\label{eq2.6}
\int_{\R^N}\big(V_1(x)u^{2}+V_2(x)v^{2}\big)dx=0.
\end{equation}
This means that $(u,v)$ is a least energy solution of \eqref{eq2.3}. By
Lemma \ref{lem2.3}, we know $u(x)>0$ and $v(x)>0$ for $x\in\R^N$. Combining
this with \eqref{eq2.6}, we have $V_1=V_2=0$ almost everywhere in $\R^N$,
which contradicts \eqref{eq2.5}. Therefore, the infimum $c$ is not attained.
\end{proof}

By Lemma \ref{lem2.4}, we know that \eqref{eq2.1} does not have a ground
state solution. Therefore, nontrivial solutions of \eqref{eq2.1} only at high
energy levels can be expected.

\section{Uniqueness for a limit problem}
This section is devoted to the classification of positive solutions for critical coupled Hartree system
\begin{equation}\label{eq3.1}
\left\{\begin{array}{ll}
-\Delta u=\alpha_1\big(|x|^{-4}\ast u^{2}\big)u+\beta \big(|x|^{-4}\ast v^{2}\big)u
&\mbox{in}\ \R^N,\\[1mm]
-\Delta v=\alpha_2\big(|x|^{-4}\ast v^{2}\big)v +\beta\big(|x|^{-4}\ast u^{2}\big)v
&\mbox{in}\ \R^N.
\end{array}\right.
\end{equation}
We will employ the Kelvin transformation and the method
of moving spheres in integral forms to complete the proof. See
\cite{CLO, DY, Lei} and references therein for uniqueness results for a
single elliptic equation.

Recall that $k_{0}=\frac{\beta-\alpha_2}{\beta^{2}-\alpha_1\alpha_2}$
and $l_{0}=\frac{\beta-\alpha_1}{\beta^{2}-\alpha_1\alpha_2}$ in Lemma \ref{lem2.3}. We also
denote $R_{N}=\frac14 \pi^{-\frac{N}{2}}\Gamma(\frac{N-2}{2})$ and, for
$0<s<\frac{N}{2}$,
$$
I(s)=\frac{\pi^{\frac{N}{2}}\Gamma(\frac{N-2s}{2})}{\Gamma(N-s)},
$$
where $\Gamma(s)=\int_0^{+\infty} x^{s-1}e^{-x}\,dx$, $s>0$.

According to \cite{CLO}, \eqref{eq3.1} is equivalent to the following
integral system in $\R^N$
\begin{equation}\label{eq3.2}
\left\{\begin{array}{l}
\displaystyle  u(x)
=\alpha_1R_{N}\int_{\R^N}\frac{u(y)w(y)}{|x-y|^{N-2}}dy
+\beta R_{N}\int_{\R^N}\frac{u(y)g(y)}{|x-y|^{N-2}}dy,\\[3mm]
\displaystyle v(x)
=\alpha_2R_{N}\int_{\R^N}\frac{v(y)g(y)}{|x-y|^{N-2}}dy
+\beta R_{N}\int_{\R^N}\frac{v(y)w(y)}{|x-y|^{N-2}}dy,\\[3mm]
\displaystyle  w(x)=\int_{\R^N}\frac{u^{2}(y)}{|x-y|^{4}}dy,\\[3mm]
\displaystyle  g(x)=\int_{\R^N}\frac{v^{2}(y)}{|x-y|^{4}}dy.
\end{array}
\right.
\end{equation}
Let $x_0\in\R^N$ and $\lambda> 0$. The inversion of
$x\in\R^N\backslash\{x_0\}$ about the sphere $\partial B_{\lambda}(x_0)$
is given by
$$
x_{x_0,\lambda}=\frac{\lambda^{2}(x-x_0)}{|x-x_0|^{2}}+x_0.
$$
Assume $(u, v, w, g)$ satisfies \eqref{eq3.2} and each component is
positive. We define the Kelvin transform of $u$ and $v$ with respect to
$\partial B_{\lambda}(x_0)$ by
$$
u_{x_0,\lambda}(x)=\Big(\frac{\lambda}{|x-x_0|}\Big)^{N-2}u(x_{x_0,\lambda}),\ \
v_{x_0,\lambda}(x)=\Big(\frac{\lambda}{|x-x_0|}\Big)^{N-2}v(x_{x_0,\lambda})
$$
and the Kelvin transform of $w$ and $g$ with respect to
$\partial B_{\lambda}(x_0)$ by
$$
w_{x_0,\lambda}(x)=\Big(\frac{\lambda}{|x-x_0|}\Big)^{4}w(x_{x_0,\lambda}),\ \
g_{x_0,\lambda}(x)=\Big(\frac{\lambda}{|x-x_0|}\Big)^{4}g(x_{x_0,\lambda}),
$$
respectively. Set
$$
U_{x_0,\lambda}=u_{x_0,\lambda}-u,\ \ V_{x_0,\lambda}=v_{x_0,\lambda}-v,\ \
W_{x_0,\lambda}=w_{x_0,\lambda}-w,\ \ G_{x_0,\lambda}=g_{x_0,\lambda}-g.
$$
When $x_0 = 0$, we will drop $x_0$ in the subscript of above notations and
write, for example,
$$
x_{\lambda}=\frac{\lambda^{2}x}{|x|^{2}},
\ \ u_{\lambda}(x)=\left(\frac{\lambda}{|x|}\right)^{N-2}u(x_{\lambda}),
\ \ U_{\lambda}=u_{\lambda}-u.
$$
Denoting
$$
\begin{array}{ll}
B_{\lambda}^{u}=\{x\in B_{\lambda}\backslash\{0\}:U_{\lambda}(x)<0\},
&B_{\lambda}^{v}=\{x\in B_{\lambda}\backslash\{0\}:V_{\lambda}(x)<0\},\\[2mm]
B_{\lambda}^{w}=\{x\in B_{\lambda}\backslash\{0\}:W_{\lambda}(x)<0\},
&B_{\lambda}^{g}=\{x\in B_{\lambda}\backslash\{0\}:G_{\lambda}(x)<0\},
\end{array}
$$
we have

\begin{lem}\label{lem3.3}
There exists a positive constant $C$ independent of $\lambda$ such that
\begin{equation}\label{eq3.3}
\aligned
|U_{\lambda}|_{2^\ast, B_{\lambda}^{u}}
\leq C\big(\alpha_1|u|^{2}_{2^\ast,B_{\lambda}^{u}}
&+\alpha_1|w|_{\frac{N}{2}, B_{\lambda}^{u}}
+\beta|g|_{\frac{N}{2}, B_{\lambda}^{u}}\big)
|U_{\lambda}|_{2^\ast, B_{\lambda}^{u}}\\[1mm]
&+C\beta|u|_{2^\ast, B_{\lambda}^{u}}|v|_{2^\ast, B_{\lambda}^{v}}
|V_{\lambda}|_{2^\ast, B_{\lambda}^{v}}
\endaligned
\end{equation}
and
\begin{equation}\label{eq3.4}
\aligned
|V_{\lambda}|_{2^\ast, B_{\lambda}^{v}}
\leq C\big(\alpha_2|v|^{2}_{2^\ast, B_{\lambda}^{v}}
&+\alpha_2|g|_{\frac{N}{2}, B_{\lambda}^{v}}
+\beta|w|_{\frac{N}{2},B_{\lambda}^{v}}\big)
|V_{\lambda}|_{2^\ast, B_{\lambda}^{v}}\\[1mm]
&+C\beta|u|_{2^\ast, B_{\lambda}^{u}}|v|_{2^\ast, B_{\lambda}^{v}}
|U_{\lambda}|_{2^\ast, B_{\lambda}^{u}},
\endaligned
\end{equation}
where $2^\ast=\frac{2N}{N-2}$.
\end{lem}

\begin{proof}[\bf Proof.]
Since $dy_\lambda=(\frac{\lambda}{|y|})^{2N}dy$, a direct computation
shows that
$$
\aligned
u(x)
&=\alpha_1R_{N}\Big(\int_{B_{\lambda}}\frac{u(y)w(y)}{|x-y|^{N-2}}dy
+\int_{\R^N\backslash B_{\lambda}}\frac{u(y)w(y)}{|x-y|^{N-2}}dy\Big)\\
&\quad\,+\beta R_{N}\Big(\int_{B_{\lambda}}\frac{u(y)g(y)}{|x-y|^{N-2}}dy
+\int_{\R^N\backslash B_{\lambda}}\frac{u(y)g(y)}{|x-y|^{N-2}}dy\Big)\\
&=\alpha_1R_{N}\Big(\int_{B_{\lambda}}\frac{u(y)w(y)}{|x-y|^{N-2}}dy
+\int_{B_{\lambda}}\frac{u_{\lambda}(y)w_{\lambda}(y)}
{\big|\frac{|y|}{\lambda}x-\frac{\lambda}{|y|}y\big|^{N-2}}dy\Big)\\
&\quad\,+\beta R_{N}\Big(\int_{B_{\lambda}}\frac{u(y)g(y)}{|x-y|^{N-2}}dy
+\int_{ B_{\lambda}}\frac{u_{\lambda}(y)g_{\lambda}(y)}
{\big|\frac{|y|}{\lambda}x-\frac{\lambda}{|y|}y\big|^{N-2}}dy\Big)
\endaligned
$$
and
$$
\aligned
u_{\lambda}(x)
&=\Big(\frac{\lambda}{|x|}\Big)^{N-2}u(x_{\lambda})\\
&=\alpha_1R_{N}\Big(\frac{\lambda}{|x|}\Big)^{N-2}
\Big(\int_{B_{\lambda}}\frac{u(y)w(y)}{|x_{\lambda}-y|^{N-2}}dy
+\int_{\R^N\backslash B_{\lambda}}\frac{u(y)w(y)}{|x_{\lambda}-y|^{N-2}}dy\Big)\\
&\quad\,+\beta R_{N}\Big(\frac{\lambda}{|x|}\Big)^{N-2}
\Big(\int_{B_{\lambda}}\frac{u(y)g(y)}{|x_{\lambda}-y|^{N-2}}dy
+\int_{\R^N\backslash B_{\lambda}}\frac{u(y)g(y)}{|x_{\lambda}-y|^{N-2}}dy\Big)\\
&=\alpha_1R_{N}\Big(\int_{B_{\lambda}}\frac{u(y)w(y)}
{\big|\frac{|y|}{\lambda}x-\frac{\lambda}{|y|}y\big|^{N-2}}dy
+\int_{B_{\lambda}}\frac{u_{\lambda}(y)w_{\lambda}(y)}{|x-y|^{N-2}}dy\Big)\\
&\quad\,+\beta R_{N}\Big(\int_{ B_{\lambda}}\frac{u(y)g(y)}
{\big|\frac{|y|}{\lambda}x-\frac{\lambda}{|y|}y\big|^{N-2}}dy
+\int_{B_{\lambda}}\frac{u_{\lambda}(y)g_{\lambda}(y)}{|x-y|^{N-2}}dy\Big),
\endaligned
$$
where we have used
$$
\left|x_{\lambda}-y\right|
=\frac{\lambda^{2}|x-y_{\lambda}|}{|x||y_{\lambda}|}
=\frac{|y|}{|x|}\bigg|x-\frac{\lambda^2}{|y|^2}y\bigg|
=\frac{\lambda}{|x|}\bigg|\frac{|y|}{\lambda}x-\frac{\lambda}{|y|}y\bigg|
$$
and
$$
|x_{\lambda}-y_{\lambda}|=\frac{\lambda^{2}|x-y|}{|x||y|}
$$
for $x,\,y \in\R^{N}\backslash\{0\}$. Then it follows
\begin{equation}\label{eq3.5}
\aligned
U_{\lambda}(x)
&=u_{\lambda}(x)-u(x)\\
&=\alpha_1R_{N}\int_{B_{\lambda}}\Big(\frac{1}{|x-y|^{N-2}}
-\frac{1}{\big|\frac{|y|}{\lambda}x-\frac{\lambda}{|y|}y\big|^{N-2}}\Big)
\times(u_{\lambda}(y)w_{\lambda}(y)-u(y)w(y))dy\\
&\quad\,+\beta R_{N}\int_{B_{\lambda}}\Big(\frac{1}{|x-y|^{N-2}}
-\frac{1}{\big|\frac{|y|}{\lambda}x-\frac{\lambda}{|y|}y\big|^{N-2}}\Big)
\times(u_{\lambda}(y)g_{\lambda}(y)-u(y)g(y))dy.
\endaligned
\end{equation}

Denoting $a^-=\min\{a, 0\}$, we claim
\begin{equation}\label{eq3.6}
u_{\lambda}w_{\lambda}-uw\geq wU_{\lambda}^{-}+uW_{\lambda}^{-}\ \ \text{and}\ \
u_{\lambda}g_{\lambda}-ug\geq gU_{\lambda}^{-}+uG_{\lambda}^{-}.
\end{equation}
Since the argument is analogous, we only prove the first inequality and
divide the discussion into four cases.

\noindent{\bf Case 1.} If $u_{\lambda}(y)\geq u(y)$ and
$w_{\lambda}(y)\geq w(y)$, then $U_{\lambda}^{-}(y)=W_{\lambda}^{-}(y)=0$
and so
$$
u_{\lambda}(y)w_{\lambda}(y)-u(y)w(y)
\geq0
=w(y)U_{\lambda}^{-}(y)+u(y)W_{\lambda}^{-}(y).
$$

\noindent{\bf Case 2.} If $u_{\lambda}(y)\geq u(y)$ and
$w_{\lambda}(y)< w(y)$, then $U_{\lambda}^{-}(y)=0$ and
$W_{\lambda}^{-}(y)=W_{\lambda}(y)$, which implies
that
$$
u_{\lambda}(y)w_{\lambda}(y)-u(y)w(y)
\geq u(y)W_{\lambda}(y)
=w(y)U_{\lambda}^{-}(y)+u(y)W_{\lambda}^{-}(y).
$$

\noindent{\bf Case 3.} If $u_{\lambda}(y)< u(y)$ and
$w_{\lambda}(y)\geq w(y)$, then $U_{\lambda}^{-}(y)=U_{\lambda}(y)$
and $W_{\lambda}^{-}(y)=0$, which implies that
$$
u_{\lambda}(y)w_{\lambda}(y)-u(y)w(y)
\geq w(y)U_{\lambda}(y)
=w(y)U_{\lambda}^{-}(y)+u(y)W_{\lambda}^{-}(y).
$$

\noindent{\bf Case 4.} If $u_{\lambda}(y)< u(y)$ and
$w_{\lambda}(y)< w(y)$, then $U_{\lambda}^{-}(y)=U_{\lambda}(y)\leq 0$
and $W_{\lambda}^{-}(y)=W_{\lambda}(y)\leq 0$. Hence
\begin{align*}
u_{\lambda}(y)w_{\lambda}(y)-u(y)w(y)
&=w_{\lambda}(y)U_{\lambda}(y)+u(y)W_{\lambda}(y)\\[1mm]
&\geq w(y)U_{\lambda}(y)+u(y)W_{\lambda}(y)
=w(y)U_{\lambda}^{-}(y)+u(y)W_{\lambda}^{-}(y).
\end{align*}

Using \eqref{eq3.5}, \eqref{eq3.6} and the fact that
$$
\left|\frac{|y|}{\lambda} x-\frac{\lambda}{|y|} y\right|^{2}-|x-y|^{2}
=\frac{\left(|x|^{2}-\lambda^{2}\right)\left(|y|^{2}-\lambda^{2}\right)}{\lambda^{2}}
>0,\ \ \text {for}\ x,\,y \in B_\lambda\backslash\{0\}
$$
leads to
$$
\aligned
U_{\lambda}(x)
&=\alpha_1R_{N}\int_{B_{\lambda}}\Big(\frac{1}{|x-y|^{N-2}}
-\frac{1}{\big|\frac{|y|}{\lambda}x-\frac{\lambda}{|y|}y\big|^{N-2}}\Big)
\times(u_{\lambda}(y)w_{\lambda}(y)-u(y)w(y))dy\\
&\quad\,+\beta R_{N}\int_{B_{\lambda}}\Big(\frac{1}{|x-y|^{N-2}}
-\frac{1}{\big|\frac{|y|}{\lambda}x-\frac{\lambda}{|y|}y\big|^{N-2}}\Big)
\times(u_{\lambda}(y)g_{\lambda}(y)-u(y)g(y))dy\\
&\geq\alpha_1R_{N}\int_{B_{\lambda}}\Big(\frac{1}{|x-y|^{N-2}}
-\frac{1}{\big|\frac{|y|}{\lambda}x-\frac{\lambda}{|y|}y\big|^{N-2}}\Big)
\times(w(y)U_{\lambda}^{-}(y)+u(y)W_{\lambda}^{-}(y))dy\\
&\quad\,+\beta R_{N}\int_{B_{\lambda}}\Big(\frac{1}{|x-y|^{N-2}}
-\frac{1}{\big|\frac{|y|}{\lambda}x-\frac{\lambda}{|y|}y\big|^{N-2}}\Big)
\times(g(y)U_{\lambda}^{-}(y)+u(y)G_{\lambda}^{-}(y))dy\\
&\geq \alpha_1R_{N}\Big(\int_{B_{\lambda}}\frac{w(y)U_{\lambda}^{-}(y)}{|x-y|^{N-2}}dy
+\int_{B_{\lambda}}\frac{u(y)W_{\lambda}^{-}(y)}{|x-y|^{N-2}}dy\Big)\\
&\quad\,+\beta R_{N}\Big(\int_{B_{\lambda}}\frac{g(y)U_{\lambda}^{-}(y)}{|x-y|^{N-2}}dy
+\int_{B_{\lambda}}\frac{u(y)G_{\lambda}^{-}(y)}{|x-y|^{N-2}}dy\Big).
\endaligned
$$
Invoking the Hardy-Littlewood-Sobolev inequality and the H\"{o}lder
inequality, we have
\begin{equation}\label{eq3.7}
\aligned
|U_{\lambda}|_{2^\ast, B_{\lambda}^{u}}
&\leq C\alpha_1|wU_{\lambda}|_{\frac{2N}{N+2}, B_{\lambda}^{u}}
+C\alpha_1|uW_{\lambda}|_{\frac{2N}{N+2}, B_{\lambda}^{u}\cap B_{\lambda}^{w}}\\[1mm]
&\quad\,+ C\beta|gU_{\lambda}|_{\frac{2N}{N+2}, B_{\lambda}^{u}}
+C\beta|uG_{\lambda}|_{\frac{2N}{N+2}, B_{\lambda}^{u}\cap B_{\lambda}^{g}}\\[1mm]
&\leq C\alpha_1|w|_{\frac{N}{2}, B_{\lambda}^{u}}|U_{\lambda}|_{2^\ast, B_{\lambda}^{u}}
+C\alpha_1|u|_{2^\ast, B_{\lambda}^{u}}|W_{\lambda}|_{\frac{N}{2}, B_{\lambda}^{w}}\\[1mm]
&\quad\,+C\beta|g|_{\frac{N}{2}, B_{\lambda}^{u}}|U_{\lambda}|_{2^\ast, B_{\lambda}^{u}}
+C\beta|u|_{2^\ast, B_{\lambda}^{u}}|G_{\lambda}|_{\frac{N}{2}, B_{\lambda}^{g}}.
\endaligned
\end{equation}

By a similar argument, we can deduce from \eqref{eq3.2} that
\begin{equation}\label{eq3.8}
W_{\lambda}(x)=\int_{B_{\lambda}}\Big(\frac{1}{|x-y|^{4}}
-\frac{1}{|\frac{|y|}{\lambda}x-\frac{\lambda}{|y|}y|^{4}}\Big)
\times(u_{\lambda}^{2}(y)-u^{2}(y))dy
\geq C\int_{B_{\lambda}}\frac{u(y)U_{\lambda}^{-}(y)}{|x-y|^{4}}dy
\end{equation}
and
\begin{equation}\label{eq3.9}
G_{\lambda}(x)=\int_{B_{\lambda}}\Big(\frac{1}{|x-y|^{4}}
-\frac{1}{|\frac{|y|}{\lambda}x-\frac{\lambda}{|y|}y|^{4}}\Big)
\times(v_{\lambda}^{2}(y)-v^{2}(y))dy
\geq C\int_{B_{\lambda}}\frac{v(y)V_{\lambda}^{-}(y)}{|x-y|^{4}}dy.
\end{equation}
Using the Hardy-Littlewood-Sobolev inequality and the H\"{o}lder inequality
again leads to
\begin{equation}\label{eq3.10}
|W_{\lambda}|_{\frac{N}{2}, B_{\lambda}^{w}}
\leq C|uU_{\lambda}|_{\frac{N}{N-2}, B_{\lambda}^{w}\cap B_{\lambda}^{u}}
\leq C|u|_{2^\ast, B_{\lambda}^{u}}|U_{\lambda}|_{2^\ast, B_{\lambda}^{u}}
\end{equation}
and
\begin{equation}\label{eq3.11}
|G_{\lambda}|_{\frac{N}{2}, B_{\lambda}^{g}}
\leq C |vV_{\lambda}|_{\frac{N}{N-2}, B_{\lambda}^{g}\cap B_{\lambda}^{v}}
\leq C |v|_{2^\ast, B_{\lambda}^{v}}|V_{\lambda}|_{2^\ast, B_{\lambda}^{v}}.
\end{equation}
Then \eqref{eq3.3} follows easily from \eqref{eq3.7}, \eqref{eq3.10} and \eqref{eq3.11}.
Similarly, one can prove \eqref{eq3.4}.
\end{proof}

\begin{lem}\label{lem3.4}
For any $x_0\in\R^N$, the sets
$$
\Gamma_{x_0}^u:=\{\lambda>0 : U_{x_0,\sigma}\geq 0\
\mbox{in}\ B_{\sigma}(x_0)\backslash\{x_0\} \ \mbox{for all }\sigma\in(0,\lambda]\}
$$
and
$$
\Gamma_{x_0}^v:=\{\lambda>0 : V_{x_0,\sigma}\geq0\
\mbox{in}\ B_{\sigma}(x_0)\backslash\{x_0\} \ \mbox{for all }\sigma\in(0,\lambda]\}
$$
are not empty.
\end{lem}

\begin{proof}[\bf Proof.]
Since \eqref{eq3.2} is invariant by translations, it suffices to consider
the case $x_0 = 0$. Let $C>0$ be the constant given in Lemma \ref{lem3.3}
and choose $\varepsilon_{0}\in(0, 1)$ sufficiently small such that, for all
$0 <\lambda\leq\varepsilon_{0}$,
$$
\alpha_1|u|^{2}_{2^\ast, B_{\lambda}^{u}}
+\alpha_1|w|_{\frac{N}{2}, B_{\lambda}^{u}}
+\beta|g|_{\frac{N}{2}, B_{\lambda}^{u}}
+\beta|u|_{2^\ast,B_{\lambda}^{u}}|v|_{2^\ast,B_{\lambda}^{v}}
\leq\frac{1}{4C}
$$
and
$$
\alpha_2|v|^{2}_{2^\ast, B_{\lambda}^{v}}
+\alpha_2|g|_{\frac{N}{2}, B_{\lambda}^{v}}
+\beta|w|_{\frac{N}{2}, B_{\lambda}^{v}}
+\beta|u|_{2^\ast,B_{\lambda}^{u}}|v|_{2^\ast, B_{\lambda}^{v}}
\leq\frac{1}{4C}.
$$
We see from Lemma \ref{lem3.3} that
$$
|U_{\lambda}|_{2^\ast, B_{\lambda}^{u}}
=|V_{\lambda}|_{2^\ast, B_{\lambda}^{v}}=0,
$$
which implies meas\,$(B_{\lambda}^{u})=$ meas\,$(B_{\lambda}^{v})=0$.
Then it follows from \eqref{eq3.5}, \eqref{eq3.8} and \eqref{eq3.9} that
$U_{\lambda}\geq0$ in $B_{\lambda}\backslash\{0\}$. Similarly, we also
have $V_{\lambda}\geq0$ in $B_{\lambda}\backslash\{0\}$. Therefore, we
have $\Gamma_0^u\neq\emptyset\neq\Gamma_0^v$.
\end{proof}

According to Lemma \ref{lem3.4}, we define
$$
\lambda_{x_0}^u:=\sup \Gamma_{x_0}^u>0,\ \
\lambda_{x_0}^v:=\sup \Gamma_{x_0}^v>0,\ \
\lambda_{x_0}:=\min\{\lambda_{x_0}^u, \lambda_{x_0}^v\}>0.
$$

\begin{lem}\label{lem3.5}
If $\lambda_{x_0}<+\infty$, then
$U_{x_0,\lambda_{x_0}}=V_{x_0,\lambda_{x_0}}=0$
in $B_{\lambda_{x_0}}(x_0)\backslash\{x_0\}$.
\end{lem}

\begin{proof}[\bf Proof.]
As before, it is sufficient to consider the case $x_0=0$. Without loss of
generality, we assume that $\lambda_0=\lambda_0^u<+\infty$. Since
$U_{\lambda}$ and $V_{\lambda}$ are continuous with respect to
$\lambda$, we have $U_{\lambda_{0}}\geq0$ and $V_{\lambda_{0}}\geq0$
in $B_{\lambda_{0}}\backslash\{0\}$. Then we see from \eqref{eq3.8} and
\eqref{eq3.9} that
$$
\aligned
W_{\lambda_{0}}(x)
&=\int_{B_{\lambda_{0}}}\Big(\frac{1}{|x-y|^{4}}
-\frac{1}{|\frac{|y|}{\lambda_{0}}x-\frac{\lambda_{0}}{|y|}y|^{4}}\Big)
\times(u_{\lambda_{0}}^{2}(y)-u^{2}(y))dy\geq 0
\endaligned
$$
and
\begin{equation}\label{eq3.12}
\aligned
G_{\lambda_{0}}(x)
&=\int_{B_{\lambda_{0}}}\Big(\frac{1}{|x-y|^{4}}
-\frac{1}{|\frac{|y|}{\lambda_{0}}x-\frac{\lambda_{0}}{|y|}y|^{4}}\Big)
\times(v_{\lambda_{0}}^{2}(y)-v^{2}(y))dy\geq0.
\endaligned
\end{equation}
Observe that if $U_{\lambda_{0}}=0$ in $B_{\lambda_0}\backslash\{0\}$,
then it follows from \eqref{eq3.5} that $G_{\lambda_0}=0$ in
$B_{\lambda_0}\backslash\{0\}$, which combined with \eqref{eq3.12} implies
$V_{\lambda_0}=0$ in $B_{\lambda_0}\backslash\{0\}$. Analogously, if
$V_{\lambda_{0}}=0$ in $B_{\lambda_0}\backslash\{0\}$, then we also
have $U_{\lambda_0}=0$ in $B_{\lambda_0}\backslash\{0\}$.

Assume to the contrary that $U_{\lambda_{0}}\not\equiv0$ and
$V_{\lambda_{0}}\not\equiv0$ in $B_{\lambda_{0}}\backslash\{0\}$. Then
$W_{\lambda_0}>0$ in $B_{\lambda_{0}}\backslash\{0\}$, which together
with \eqref{eq3.5} indicates that $U_{\lambda_{0}}>0$ in
$B_{\lambda_{0}}\backslash\{0\}$. Similarly, there also holds
$V_{\lambda_{0}}>0$ in $B_{\lambda_{0}}\backslash\{0\}$. We claim that
there exist two numbers $\rho>0$ and $\delta>0$ such that
\begin{equation}\label{eq3.13}
U_{\lambda_{0}}\geq \delta\ \ \text{and}\ \ V_{\lambda_{0}}\geq \delta \ \mbox{ in } B_{\rho}\backslash\{0\}.
\end{equation}
Indeed, since $U_{\lambda_{0}}>0$, $V_{\lambda_{0}}>0$,
$W_{\lambda_{0}}>0$ and $G_{\lambda_{0}}>0$ in
$B_{\lambda_{0}}\backslash\{0\}$, we deduce from \eqref{eq3.5} that
$$
\aligned
\liminf_{|x|\rightarrow0}U_{\lambda_{0}}(x)
&\geq\alpha_1R_{N}\int_{B_{\lambda_{0}}}\left(\frac{1}{|y|^{N-2}}
-\frac{1}{\lambda_{0}^{N-2}}\right)
\times(u_{\lambda_{0}}(y)w_{\lambda_{0}}(y)-u(y)w(y))dy\\
&\quad\,+\beta R_{N}\int_{B_{\lambda_{0}}}\left(\frac{1}{|y|^{N-2}}
-\frac{1}{\lambda_{0}^{N-2}}\right)
\times(u_{\lambda_{0}}(y)g_{\lambda_{0}}(y)-u(y)g(y))dy\\
&>0
\endaligned
$$
and, similarly, $\liminf_{|x|\rightarrow0}V_{\lambda_{0}}(x)>0$.
Then \eqref{eq3.13} holds for small $\rho>0$ and $\delta > 0$.

Let $C>0$ be the constant given in Lemma \ref{lem3.3} and fix a number
$r_0\in (0, \frac{\lambda_{0}}{2})$ such that
\begin{equation}\label{eq3.14}
\aligned
\alpha_1|u|^{2}_{2^\ast, B_{\lambda_{0}+r_{0}}\backslash B_{\lambda_{0}-r_{0}}}
+\alpha_1&|w|_{\frac{N}{2},B_{\lambda_{0}+r_{0}}\backslash B_{\lambda_{0}-r_{0}}}
+\beta|g|_{\frac{N}{2}, B_{\lambda_{0}+r_{0}}\backslash B_{\lambda_{0}-r_{0}}}\\
&+\beta|u|_{2^\ast, B_{\lambda_{0}+r_{0}}\backslash B_{\lambda_{0}-r_{0}}}
|v|_{2^\ast, B_{\lambda_{0}+r_{0}}\backslash B_{\lambda_{0}-r_{0}}}
\leq\frac{1}{4C}
\endaligned
\end{equation}
and
\begin{equation}\label{eq3.15}
\aligned
\alpha_2|v|^{2}_{2^\ast, B_{\lambda_{0}+r_{0}}\backslash B_{\lambda_{0}-r_{0}}}
+\alpha_2&|g|_{\frac{N}{2},B_{\lambda_{0}+r_{0}}\backslash B_{\lambda_{0}-r_{0}}}
+\beta|w|_{\frac{N}{2}, B_{\lambda_{0}+r_{0}}\backslash B_{\lambda_{0}-r_{0}}}\\
&+\beta|u|_{2^\ast, B_{\lambda_{0}+r_{0}}\backslash B_{\lambda_{0}-r_{0}}}
|v|_{2^\ast, B_{\lambda_{0}+r_{0}}\backslash B_{\lambda_{0}-r_{0}}}
\leq\frac{1}{4C}.
\endaligned
\end{equation}
Since $U_{\lambda_{0}}>0$ and $V_{\lambda_{0}}>0$ in $B_{\lambda_{0}}\backslash\{0\}$,
we see from \eqref{eq3.13} that $U_{\lambda_{0}}\geq \delta'$ and $V_{\lambda_{0}}\geq \delta'$
in $B_{\lambda_{0}-r_{0}}\backslash\{0\}$ for some $\delta'>0$. By uniform
continuity, there exists $\tau_{0}\in(0, r_0)$ such that, for any
$\lambda\in(\lambda_{0},\lambda_{0}+\tau_{0})$, $U_{\lambda}\geq \delta'/2$
and $V_{\lambda}\geq \delta'/2$ in $B_{\lambda_{0}-r_{0}}\backslash\{0\}$. Then
$B_{\lambda}^{u}\subset B_{\lambda_{0}+r_{0}}\backslash B_{\lambda_{0}-r_{0}}$
and $B_{\lambda}^{v}\subset B_{\lambda_{0}+r_{0}}\backslash B_{\lambda_{0}-r_{0}}$
for any $\lambda\in(\lambda_{0},\lambda_{0}+\tau_{0})$. Using Lemma \ref{lem3.3},
\eqref{eq3.14} and \eqref{eq3.15} leads to $|U_{\lambda}|_{2^\ast, B_{\lambda}^{u}}=0$
when $\lambda\in(\lambda_{0},\lambda_{0}+\tau_{0})$, which means
meas\,$(B_{\lambda}^{u})=0$. Therefore, we have $U_{\lambda}\geq 0$ in
$B_{\lambda}\backslash\{0\}$ for any $\lambda\in(\lambda_{0},\lambda_{0}+\tau_{0})$,
which contradicts the definition of $\lambda_{0}^u$. The proof is complete.
\end{proof}

The following lemma is proved in \cite{LZ1, Liyy}.

\begin{lem} \label{lem3.6}
Let $N\geq1$, $\nu\in\R$ and $u\in C^{1}(\R^{N}, \mathbb R)$. For every
$x_0\in\R^N$ and $\lambda> 0$, define
$$
u_{x_0,\lambda}(x)=\Big(\frac{\lambda}{|x-x_0|}\Big)^{\nu}
u\Big(\frac{\lambda^{2}(x-x_0)}{|x-x_0|^{2}}+x_0\Big),\ \
x\in\R^N\backslash\{x_0\}.
$$
$(i)$ If for every $x_0\in\R^N$ there exists $\lambda_{x_0}<+\infty$ such
that
$$
u_{x_0,\lambda_{x_0}}(x)=u(x),\ \ \text{for all}\ x\in\R^N\backslash\{x_0\},
$$
then
$$
u(x)=C\Big(\frac{\tau}{\tau^{2}+|x-\overline{x}|^{2}}\Big)^{\frac{\nu}{2}}
$$
for some $C\in\R$, $\tau>0$ and $\overline{x}\in\R^N$.
Moreover, we have $\lambda_{x_0}=\sqrt{\tau^{2}+|x_{0}-\overline{x}|^{2}}$.

\vskip 1mm
\noindent $(ii)$ If for every $x_0\in\R^N$ there holds
$$
u_{x_0,\lambda}(x)\geq u(x),\ \ \mbox{for all}\ \lambda\in\R\
\text{and}\ x\in B_{\lambda}(x_0)\backslash\{x_0\},
$$
then $u\equiv C$ for some $C\in\R$.
\end{lem}

We are in a position to prove the main result in this section.

\begin{proof}[\bf Proof of Theorem \ref{thm3.1+}.]
Let $(u,v)$ be a positive solution of \eqref{eq3.1} and recall that
$$
\lambda_{x_0}^u:=\sup \Gamma_{x_0}^u>0,\ \
\lambda_{x_0}^v:=\sup \Gamma_{x_0}^v>0,\ \
\lambda_{x_0}:=\min\{\lambda_{x_0}^u, \lambda_{x_0}^v\}>0,
$$
where
$$
\Gamma_{x_0}^u:=\{\lambda>0 : U_{x_0,\sigma}\geq 0
\mbox{ in }B_{\sigma}(x_0)\backslash\{x_0\} \ \mbox{for all }\sigma\in(0,\lambda]\}
$$
and
$$
\Gamma_{x_0}^v:=\{\lambda>0 : V_{x_0,\sigma}\geq0 \mbox{ in }
B_{\sigma}(x_0)\backslash\{x_0\} \ \mbox{for all }\sigma\in(0,\lambda]\}.
$$
We first claim that $\lambda_{x_0}<+\infty$ for any $x_0\in\R^N$. If not,
then we have the following two cases.

\textbf{Case 1}. $\lambda_{x_0}=+\infty$ for any $x_0\in\R^N$. In this
case, we see from Lemma \ref{lem3.6}$(ii)$ that
$$
(u, v) \equiv (C_1, C_2)
$$
for some constants $C_1,\,C_2 > 0$. Then $(u, v)$ could not be a solution
of \eqref{eq3.1}, yielding a contradiction.

\textbf{Case 2}. There exist $x_0,\,y_{0}\in\R^N$ such that
$\lambda_{x_0}=+\infty$ and $\lambda_{y_0}<+\infty$. In this case, since
$\lambda_{x_0}^u\geq \lambda_{x_0}=+\infty$, we have, for any $\lambda>0$,
$U_{x_0, \lambda}\geq 0$ for $x\in B_{\lambda}(x_0)\backslash\{x_0\}$
which implies that $u(x)\geq u_{x_0, \lambda}(x)$ for
$x\in \R^N\backslash B_{\lambda}(x_0)$. Then we obtain
$|x-x_0|^{N-2}u(x)\geq\lambda^{N-2}u(x_{x_0,\lambda})$ for
$x\in \R^N\backslash B_{\lambda}(x_0)$ and so
$$
\liminf_{|x|\rightarrow\infty}|x|^{N-2}u(x)\geq\lambda^{N-2}u(x_0).
$$
Since $\lambda>0$ is arbitrary and $u(x_0)>0$, we obtain
\begin{equation}\label{eq3.16}
\lim_{|x|\rightarrow\infty}|x|^{N-2}u(x)=+\infty.
\end{equation}
On the other hand, since $\lambda_{y_0}<+\infty$, we see from Lemma
\ref{lem3.5} that
$$
u_{y_0, \lambda_{y_0}}(x)=u(x),\ \ \mbox{for}\ x\in \R^N\backslash \{y_0\}.
$$
Then we have
$\lim_{|x|\rightarrow\infty}|x|^{N-2}u(x)=\lambda_{y_0}^{N-2}u(y_0)<+\infty$,
yielding a contradiction with \eqref{eq3.16}.

Since $\lambda_{x_0}<+\infty$ for any $x_0\in\R^N$, we deduce from
Lemma \ref{lem3.5} that
$$
u_{x_0, \lambda_{x_0}}(x)=u(x)\ \ \text{and}\ \ v_{x_0, \lambda_{x_0}}(x)=v(x),\ \
\mbox{for all}\ x\in \R^N\backslash \{x_0\}.
$$
In view of Lemma \ref{lem3.6}$(i)$, $(u, v)$ must be of the form
\begin{equation}\label{eq3.17}
u(x)=C_{1}\Big(\frac{\tau}{\tau^{2}+|x-\overline{x}|^{2}}\Big)^{\frac{N-2}{2}},\ \
v(x)=C_{2}\Big(\frac{\tau}{\tau^{2}+|x-\overline{x}|^{2}}\Big)^{\frac{N-2}{2}}
\end{equation}
for some $C_{1}$, $C_{2}$, $\tau>0$ and $\overline{x}\in\R^N$.

Using \eqref{eq3.17} and the identity (see \cite[(37)]{DHQWF} for example)
\begin{equation}\label{eq3.18}
\int_{\R^N}\frac{1}{|x-y|^{2s}}\Big(\frac{1}{1+|y|^{2}}\Big)^{N-s}dy
=I(s)\Big(\frac{1}{1+|x|^{2}}\Big)^{s},\ \ 0 < s < \frac{N}{2},
\end{equation}
we have
$$
w(x)
=\int_{\R^N}\frac{u^2(y)}{|x-y|^{4}}dy
=C_{1}^{2}I(2)\Big(\frac{\tau}{\tau^{2}+|x-\overline{x}|^{2}}\Big)^{2}
$$
and
$$
g(x)
=\int_{\R^N}\frac{v^2(y)}{|x-y|^{4}}dy
=C_{2}^{2}I(2)\Big(\frac{\tau}{\tau^{2}+|x-\overline{x}|^{2}}\Big)^{2}.
$$
Then we deduce from \eqref{eq3.2} and \eqref{eq3.18} that
$$
\aligned
u(x)
&=\alpha_1R_{N}\int_{\R^N}\frac{u(y)w(y)}{|x-y|^{N-2}}dy
+\beta R_{N}\int_{\R^N}\frac{u(y)g(y)}{|x-y|^{N-2}}dy\\
&=\alpha_1R_{N}C_{1}^{3}I(2)I\Big(\frac{N-2}{2}\Big)
\Big(\frac{\tau}{\tau^{2}+|x-\overline{x}|^{2}}\Big)^{\frac{N-2}{2}}
+\beta R_{N}C_{1}C_{2}^{2}I(2)I\Big(\frac{N-2}{2}\Big)
\Big(\frac{\tau}{\tau^{2}+|x-\overline{x}|^{2}}\Big)^{\frac{N-2}{2}},
\endaligned$$
which combined with \eqref{eq3.17} leads to
$$
\alpha_1R_{N}C_{1}^{2}I(2)I\Big(\frac{N-2}{2}\Big)
+\beta R_{N}C_{2}^{2}I(2)I\Big(\frac{N-2}{2}\Big)=1.
$$
Similarly, we also have
$$
\alpha_2R_{N}C_{2}^{2}I(2)I\Big(\frac{N-2}{2}\Big)
+\beta R_{N}C_{1}^{2}I(2)I\Big(\frac{N-2}{2}\Big)=1.
$$
A simple calculation shows that
$$
C_{1}=\frac{\sqrt{k_{0}}}{\sqrt{R_{N}I(2)I\big(\frac{N-2}{2}\big)}},\ \
C_{2}=\frac{\sqrt{l_{0}}}{\sqrt{R_{N}I(2)I\big(\frac{N-2}{2}\big)}}.
$$
The proof is completed.
\end{proof}

As a direct consequence of Theorem \ref{thm3.1+} and Lemma \ref{lem2.3},
we have the following corollary.

\begin{cor}\label{cor3.2}
Let $\beta>\max\{\alpha_1,\alpha_2\}$. If $(u, v)\in H$  is a nontrivial classical positive solution of
\eqref{eq2.3}, then we have
$$
(u,v)=(\sqrt{k_{0}}U_{\delta,z},\sqrt{l_{0}}U_{\delta,z})
$$
for some $\delta>0$ and $z\in\R^N$. Moreover, each nontrivial classical positive solution $(u, v)\in H$ of
\eqref{eq2.3} is a ground state solution.
\end{cor}

\section{A nonlocal global compactness lemma}
In this section, we will prove a nonlocal global compactness result for
\eqref{eq2.1}, i.e., we will give a complete description for the Palais-Smale
sequences of the functional $J$. We start with a Br\'{e}zis-Lieb type lemma
about the nonlocal term which is inspired by the Br\'{e}zis-Lieb convergence
lemma (see \cite{BL1}). The proof is analogous to that of \cite[Lemma 2.2]{GY}
and \cite[Lemma 2.4]{MS1}, but we exhibit it here for completeness.

\begin{lem}\label{lem4.1}
Let $N\geq5$ and assume $\{(u_{n},v_{n})\}$ to be a bounded sequence
in $L^{\frac{2N}{N-2}}(\R^N)\times L^{\frac{2N}{N-2}}(\R^N)$ such that
$(u_{n},v_{n})\rightarrow (u,v)$ almost everywhere in $\R^N$ as
$n\rightarrow\infty$. Then we have
$$
\int_{\R^N}(|x|^{-4}\ast |u_{n}^{+}|^{2})|u_{n}^{+}|^{2}dx
-\int_{\R^N}(|x|^{-4}\ast |(u_{n}-u)^{+}|^{2})|(u_{n}-u)^{+}|^{2}dx
\rightarrow\int_{\R^N}(|x|^{-4}\ast |u^{+}|^{2})|u^{+}|^{2}dx
$$
and
$$
\int_{\R^N}(|x|^{-4}\ast |u_{n}^{+}|^{2})|v_{n}^{+}|^{2}dx
-\int_{\R^N}(|x|^{-4}\ast |(u_{n}-u)^{+}|^{2})|(v_{n}-v)^{+}|^{2}dx
\rightarrow\int_{\R^N}(|x|^{-4}\ast |u^{+}|^{2})|v^{+}|^{2}dx
$$
as $n\rightarrow\infty$.
\end{lem}

\begin{proof}[\bf Proof.]
Similar to the proof of the Br\'{e}zis-Lieb Lemma in \cite{BL1}, we have
\begin{equation}\label{eq4.1}
|u_{n}^{+}|^{2}-|(u_{n}-u)^{+}|^{2}\rightarrow|u^{+}|^{2}\ \ \text{in}\ L^{\frac{N}{N-2}}(\R^N)
\end{equation}
and
\begin{equation}\label{eq4.2}
|v_{n}^{+}|^{2}-|(v_{n}-v)^{+}|^{2}\rightarrow|v^{+}|^{2}\ \ \text{in}\ L^{\frac{N}{N-2}}(\R^N).
\end{equation}
Using Proposition \ref{pro1.1} yields
\begin{equation}\label{eq4.3}
|x|^{-4}\ast(|u_{n}^{+}|^{2}-|(u_{n}-u)^{+}|^{2})\rightarrow|x|^{-4}\ast|u^{+}|^{2}
\ \ \text{in}\ L^{\frac{N}{2}}(\R^N)
\end{equation}
and
\begin{equation}\label{eq4.4}
|x|^{-4}\ast(|v_{n}^{+}|^{2}-|(v_{n}-v)^{+}|^{2})\rightarrow|x|^{-4}\ast|v^{+}|^{2}
\ \ \text{in}\ L^{\frac{N}{2}}(\R^N).
\end{equation}
Note that
\begin{equation}\label{eq4.5}
\aligned
&\quad\,\int_{\R^N}\big(|x|^{-4}\ast |u_{n}^{+}|^{2}\big)|u_{n}^{+}|^{2}dx
-\int_{\R^N}\big(|x|^{-4}\ast |(u_{n}-u)^{+}|^{2}\big)|(u_{n}-u)^{+}|^{2}dx\\
&=\int_{\R^N}\big(|x|^{-4}\ast (|u_{n}^{+}|^{2}-|(u_{n}-u)^{+}|^{2})\big)
(|u_{n}^{+}|^{2}-|(u_{n}-u)^{+}|^{2})dx\\
&\quad\,+2\int_{\R^N}\big(|x|^{-4}\ast (|u_{n}^{+}|^{2}-|(u_{n}-u)^{+}|^{2})\big)
|(u_{n}-u)^{+}|^{2}dx
\endaligned
\end{equation}
and
\begin{equation}\label{eq4.6}
\aligned
&\quad\,\int_{\R^N}\big(|x|^{-4}\ast |u_{n}^{+}|^{2}\big)|v_{n}^{+}|^{2}dx
-\int_{\R^N}\big(|x|^{-4}\ast |(u_{n}-u)^{+}|^{2}\big)|(v_{n}-v)^{+}|^{2}dx\\
&=\int_{\R^N}\big(|x|^{-4}\ast (|u_{n}^{+}|^{2}-|(u_{n}-u)^{+}|^{2})\big)
(|v_{n}^{+}|^{2}-|(v_{n}-v)^{+}|^{2})dx\\
&\quad\,+\int_{\R^N}\big(|x|^{-4}\ast (|u_{n}^{+}|^{2}-|(u_{n}-u)^{+}|^{2})\big)
|(v_{n}-v)^{+}|^{2}dx\\
&\quad\,+\int_{\R^N}\big(|x|^{-4}\ast (|v_{n}^{+}|^{2}-|(v_{n}-v)^{+}|^{2})\big)
|(u_{n}-u)^{+}|^{2}dx.
\endaligned
\end{equation}
Combining \eqref{eq4.1}$-$\eqref{eq4.6} with the fact that
$$
|(u_{n}-u)^{+}|^{2}\rightharpoonup0\ \ \text{and}\ \
|(v_{n}-v)^{+}|^{2}\rightharpoonup0\ \ \text{in}\ L^{\frac{N}{N-2}}(\R^N)
$$
leads to the desired result.
\end{proof}

The global compactness lemma plays an important role in the study of critical
problems, see \cite{Sm, Wi} for a single elliptic equation with a local interaction,
\cite{LL1, PPW} for local system and \cite{GSYZ} for a nonlocal Choquard
equation. For $r\in\R^{+}$ and $z\in \R^{N}$, we denote the rescaling
$$
(u,v)_{r,z}=r^{\frac{N-2}{2}}(u(rx+z),v(rx+z)).
$$
Inspired by the above results, we can establish the global compactness
lemma for nonlocal type systems.

\begin{lem}\label{lem4.2}
Suppose that $V_1,\,V_2\in L^{\frac{N}{2}}(\R^{N})\cap L_{\text{\rm loc}}^{\infty}(\R^N)$
and $\{(u_{n},v_{n})\}\subset H$ is a $(PS)_{d}$ sequence for the functional
$J$. Then there exist a number $k\in \N$, a solution $(u^{0},v^{0})$ of \eqref{eq2.1},
nonzero solutions $(u^{1},v^{1}),\cdots,(u^{k},v^{k})$ of \eqref{eq2.3},
sequences of points $\{z_{n}^{1}\},\cdots,\{z_{n}^{k}\}$ in $\R^N$ and radii
$\{r_{n}^{1}\},\cdots,\{r_{n}^{k}\}$ such that, up to a subsequence,
$$
(u_{n}^{0},v_{n}^{0}):= (u_{n},v_{n})\rightharpoonup (u^{0},v^{0})\ \ \mbox{in}\ H
$$
and
$$
(u_{n}^{j},v_{n}^{j}):= (u_{n}^{j-1}-u^{j-1},v_{n}^{j-1}-v^{j-1})_{r_{n}^{j},z_{n}^{j}}
\rightharpoonup (u^{j},v^{j})\ \ \mbox{in}\ H,\ \ j=1,...,k.
$$
Moreover, we have
$$
\lim_{n\to\infty}\|(u_{n},v_{n})\|^{2}=\sum_{j=0}^{k}\|(u^{j},v^{j})\|^{2}
$$
and
$$
\lim_{n\to\infty}J(u_{n},v_{n})=J(u^{0},v^{0})+\sum_{j=1}^{k}J_{\infty}(u^{j},v^{j}).
$$
\end{lem}

\begin{proof}[\bf Proof.]
Let $\{(u_{n},v_{n})\}$ be a $(PS)_{d}$ sequence for $J$, then it is bounded
in $H$. We assume up to a subsequence that
$(u_{n},v_{n})\rightharpoonup (u^{0},v^{0})$ in $H$, $(u_{n},v_{n})\to (u^{0},v^{0})$
almost everywhere in $\R^N$ and $(u^{0},v^{0})$ is a weak solution of \eqref{eq2.1}.
Setting $(\overline{u}_{n}, \overline{v}_{n}):=(u_{n}-u^{0}, v_{n}-v^{0})$, we
have $(\overline{u}_{n},\overline{v}_{n})\rightharpoonup(0,0)$ in $H$. Using
this together with the Br\'{e}zis-Lieb Lemma \cite{BL1} and  Lemma \ref{lem4.1},
we deduce that
\begin{align}\label{eq4.7}
\|(\overline{u}_{n},\overline{v}_{n})\|^{2}
&=\|(u_{n},v_{n})\|^{2}-\|(u^{0},v^{0})\|^{2}+o_{n}(1),\\[1mm]
J(\overline{u}_{n},\overline{v}_{n})&=J(u_{n},v_{n})-J(u^{0},v^{0})+o_{n}(1)\notag
\end{align}
and
$$
\hspace{1.1cm}J'(\overline{u}_{n},\overline{v}_{n})
=J'(u_{n},v_{n})-J'(u^{0},v^{0})+o_{n}(1)=o_{n}(1).
$$

Since $V_1\in L^{\frac{N}{2}}(\R^N)$, for any $\varepsilon> 0$ there exists
a number $r=r(\varepsilon)> 0$ such that
$$
\Big(\int_{\R^N\backslash B_r(0)}|V_1|^{\frac{N}{2}}dx\Big)^{\frac{2}{N}}
<\varepsilon.
$$
For such an $r$, we can find $n_{0}\in\N$ such that
$$
\int_{B_r(0)}\overline{u}_{n}^{2}dx<\varepsilon,\ \ \mbox{for}\ n\geq n_0.
$$
Then, using $V_1\in L_{\text{loc}}^{\infty}(\R^N)$ and the H\"older inequality,
we have
$$
\aligned
\Big|\int_{\R^N}V_1(x)\overline{u}_{n}^{2}dx\Big|
&=\Big|\int_{B_r(0)}V_1(x)\overline{u}_{n}^{2}dx
+\int_{\R^N\backslash B_r(0)}V_1(x)\overline{u}_{n}^{2}dx\Big|\\
&\leq|V_1|_{\infty, B_r(0)}\int_{B_r(0)}\overline{u}_{n}^{2}dx+
\Big(\int_{\R^N\backslash B_r(0)}|V_1|^{\frac{N}{2}}dx\Big)^{\frac{2}{N}}
\Big(\int_{\R^{N}}|\overline{u}_{n}|^{\frac{2N}{N-2}}dx\Big)^{\frac{N-2}{N}}\\
&<C\varepsilon
\endaligned
$$
for $n\geq n_{0}$, which means that
$$
\lim_{n\rightarrow\infty}\int_{\R^N}V_1(x)\overline{u}_{n}^{2}dx=0.
$$
Similarly, we also have
$$
 \lim_{n\rightarrow\infty}\int_{\R^N}V_2(x)\overline{v}_{n}^{2}dx=0.
$$
Then
\begin{equation}\label{eq4.8}
J_{\infty}(\overline{u}_{n},\overline{v}_{n})
=J(\overline{u}_{n},\overline{v}_{n})+o_{n}(1)
=J(u_{n},v_{n})-J(u^{0},v^{0})+o_{n}(1)
\end{equation}
and
\begin{equation}\label{eq4.9}
J_{\infty}'(\overline{u}_{n},\overline{v}_{n})
=J'(\overline{u}_{n},\overline{v}_{n})+o_{n}(1)=o_{n}(1).
\end{equation}

If $(\overline{u}_{n},\overline{v}_{n})\rightarrow(0,0)$ in $H$ then we are
done: $k$ is just 0 and $(u_{n}^{0},v_{n}^{0}):=(u_{n},v_{n})$. Now we
consider the case where $(\overline{u}_{n},\overline{v}_{n})\nrightarrow(0,0)$
in $H$. Assume up to a subsequence that
$\lim_{n\to\infty}\|(\overline{u}_{n},\overline{v}_{n})\|^{2}=b>0$ and, by
\eqref{eq4.9}, there also holds
$$
\lim_{n\to\infty}\int_{\R^N}\int_{\R^N}
\frac{\alpha_1|\overline{u}_{n}^{+}(x)|^{2}
|\overline{u}_{n}^{+}(y)|^{2}+\alpha_2|\overline{v}_{n}^{+}(x)|^{2}
|\overline{v}_{n}^{+}(y)|^{2}+2\beta|\overline{u}_{n}^{+}(x)|^{2}
|\overline{v}_{n}^{+}(y)|^{2}}
{|x-y|^{4}}dxdy
=b.
$$
For $t_n > 0$ defined by
$$
t_n^{2}=\frac{\displaystyle\|(\overline{u}_{n},\overline{v}_{n})\|^{2}}
{\displaystyle\int_{\R^N}\int_{\R^N}
\frac{\alpha_1|\overline{u}_{n}^{+}(x)|^{2}
|\overline{u}_{n}^{+}(y)|^{2}+\alpha_2|\overline{v}_{n}^{+}(x)|^{2}
|\overline{v}_{n}^{+}(y)|^{2}+2\beta|\overline{u}_{n}^{+}(x)|^{2}
|\overline{v}_{n}^{+}(y)|^{2}}
{|x-y|^{4}}dxdy},
$$
we have $(t_n\overline{u}_{n},t_n\overline{v}_{n})\in\mathcal{N}_{\infty}$
and $t_n^{2}=1+o_{n}(1)$ as $n\rightarrow\infty$. Then
$$
c_{\infty}\leq J_{\infty}(t_n\overline{u}_{n},t_n\overline{v}_{n})
=\frac{1}{4}t_n^{2}\|(\overline{u}_{n},\overline{v}_{n})\|^{2}
=\frac{1}{4}b+o_{n}(1),
$$
which implies that $b\geq 4c_{\infty}$.

{\bf Claim:} There exist sequences $\{r_{n}\}\subset\R^+$ and
$\{z_{n}\}\subset\R^N$ such that
$$
(\widetilde{u}_{n},\widetilde{v}_{n})
=(\overline{u}_{n},\overline{v}_{n})_{r_{n},z_{n}}\rightharpoonup (u,v)
\ \ \mbox{in}\ H,
$$
where $(u,v)$ is a nonzero solution of \eqref{eq2.3}.

We see from \eqref{eq4.9} that
$$
J_{\infty}(\overline{u}_{n},\overline{v}_{n})
=\frac{1}{4}\|(\overline{u}_{n},\overline{v}_{n})\|^{2}+o_{n}(1).
$$
Define the Levy concentration function of $(\overline{u}_{n}, \overline{v}_{n})$
by
$$
Q_{n}(r):=\sup_{z\in\R^N}\int_{B_{r}(z)}
(|\nabla \overline{u}_{n}|^{2}+|\nabla \overline{v}_{n}|^{2})dx.
$$
Let $L$ be the least number of balls with radius 1 covering a ball of radius 2.
We see from
$$
\lim_{n\to\infty}\|(\overline{u}_{n},\overline{v}_{n})\|^{2}=b\geq 4c_{\infty}
$$
that, for large $n$, there exist $r_{n}\in\R^+$ and $z_{n}\in\R^N$ such that
$$
\sup_{z\in\R^N}\int_{B_{r}(z)}(|\nabla \overline{u}_{n}|^{2}+|\nabla \overline{v}_{n}|^{2})dx
=\int_{B_{r_{n}}(z_{n})}(|\nabla \overline{u}_{n}|^{2}+|\nabla \overline{v}_{n}|^{2})dx
=\frac{2c_{\infty}}{L}.
$$
Setting $(\widetilde{u}_{n},\widetilde{v}_{n}):=(\overline{u}_{n},\overline{v}_{n})_{r_{n},z_{n}}$,
we have
\begin{equation}\label{eq4.10}
\sup_{z\in\R^N}\int_{B_{1}(z)}(|\nabla \widetilde{u}_{n}|^{2}+|\nabla \widetilde{v}_{n}|^{2})dx
=\int_{B_{1}(0)}(|\nabla \widetilde{u}_{n}|^{2}+|\nabla \widetilde{v}_{n}|^{2})dx
=\frac{2c_{\infty}}{L}
\end{equation}
and $\{(\widetilde{u}_{n},\widetilde{v}_{n})\}$ is bounded in $H$. Assume
by extracting a subsequence that
$(\widetilde{u}_{n},\widetilde{v}_{n})\rightharpoonup (u,v)$ in $H$ and
$(\widetilde{u}_{n},\widetilde{v}_{n})\rightarrow (u,v)$ almost everywhere
in $\R^N$. The scale invariance under translation and dilation implies that
$$
\begin{aligned}
\|(\overline{u}_{n},\overline{v}_{n})\|
&=\|(\widetilde{u}_{n},\widetilde{v}_{n}\|,\\
\int_{\R^N}\int_{\R^N}\frac{|\overline{u}_{n}^{+}(x)|^2
|\overline{u}_{n}^{+}(y)|^{2}}{|x-y|^{4}}dxdy
&=\int_{\R^N}\int_{\R^N}
\frac{|\widetilde{u}_{n}^{+}(x)|^{2}|\widetilde{u}_{n}^{+}(y)|^{2}}{|x-y|^{4}}dxdy,\\
\int_{\R^N}\int_{\R^N}\frac{|\overline{v}_{n}^{+}(x)|^{2}|\overline{v}_{n}^{+}(y)|^{2}}
{|x-y|^{4}}dxdy
&=\int_{\R^N}\int_{\R^N}
\frac{|\widetilde{v}_{n}^{+}(x)|^{2}|\widetilde{v}_{n}^{+}(y)|^{2}}{|x-y|^{4}}dxdy,\\
\int_{\R^N}\int_{\R^N}\frac{|\overline{u}_{n}^{+}(x)|^{2}|\overline{v}_{n}^{+}(y)|^{2}}
{|x-y|^{4}}dxdy
&=\int_{\R^N}\int_{\R^N}
\frac{|\widetilde{u}_{n}^{+}(x)|^{2}|\widetilde{v}_{n}^{+}(y)|^{2}}{|x-y|^{4}}dxdy.
\end{aligned}
$$
Then we have
$$
J_{\infty}(\widetilde{u}_{n},\widetilde{v}_{n})
=J_{\infty}(\overline{u}_{n},\overline{v}_{n})+o_{n}(1)
$$
and
$$
\|J_{\infty}'(\widetilde{u}_{n},\widetilde{v}_{n})\|
=\|J_{\infty}'(\overline{u}_{n},\overline{v}_{n})\|=o_{n}(1).
$$
Therefore, $(u, v)$ is a solution of \eqref{eq2.3}.

Next we show that $(u, v) \neq (0, 0)$. In fact, using the arguments in
\cite{Sm}, we can find $\rho\in[1, 2]$ such that the solution
$\widehat{\varphi}_{n}$ of the boundary value problem
$$
\left\{\begin{array}{l}
-\Delta \varphi=0\ \ \mbox{in}\ B_{3}(0)\backslash B_{\rho}(0),\\[1mm]
\varphi|_{\partial B_{\rho}(0)}=\widetilde{u}_{n}-u,\
\varphi|_{\partial B_{3}(0)}=0
\end{array}
\right.
$$
satisfies $\widehat{\varphi}_{n}\rightarrow0$ in
$H^{1}(B_{3}(0)\backslash B_{\rho}(0))$ and the solution
$\widehat{\psi}_{n}$ of the problem in which the boundary condition
$\varphi|_{\partial B_{\rho}(0)}=\widetilde{u}_{n}-u$ is replaced with
$\varphi|_{\partial B_{\rho}(0)}=\widetilde{v}_{n}-v$ also satisfies
$\widehat{\psi}_{n}\rightarrow0$ in $H^{1}(B_{3}(0)\backslash B_{\rho}(0))$.
Define
$$
\widetilde{\varphi}_{n}(x)=
\left\{\begin{array}{ll}
\widetilde{u}_{n}(x)-u(x), & x\in B_{\rho}(0),\\[1mm]
\widehat{\varphi}_{n}, & x\in B_{3}(0)\backslash B_{\rho}(0),\\[1mm]
0, & x\in\R^N\backslash B_{3}(0).
\end{array}\right.
$$
Replace $\widetilde{u}_{n}$ and $\widehat{\varphi}_{n}$ with
$\widetilde{v}_{n}$ and $\widehat{\psi}_{n}$ respectively in the definition
of $\widetilde{\varphi}_{n}$, and denote this resulted new function by
$\widetilde{\psi}_{n}$. Setting
$$
\varphi_{n}=r_{n}^{-\frac{N-2}{2}}\widetilde{\varphi}_{n}\Big(\frac{\cdot-z_{n}}{r_{n}}\Big)
\ \ \mbox{and}
\ \ \psi_{n}=r_{n}^{-\frac{N-2}{2}}\widetilde{\psi}_{n}\Big(\frac{\cdot-z_{n}}{r_{n}}\Big),
$$
we have
\begin{equation}\label{eq4.11}
\aligned
\int_{\R^N}(|\nabla \varphi_{n}|^{2}+|\nabla \psi_{n}|^{2})dx
&=\int_{\R^N}(|\nabla \widetilde{\varphi}_{n}|^{2}+|\nabla \widetilde{\psi}_{n}|^{2})dx\\
&=\int_{B_{\rho}(0)}(|\nabla (\widetilde{u}_{n}-u)|^{2}+|\nabla (\widetilde{v}_{n}-v)|^{2})dx+o_{n}(1)\\
&=\int_{B_{\rho}(0)}(|\nabla \widetilde{u}_{n}|^{2}+|\nabla \widetilde{v}_{n}|^{2})dx
-\int_{B_{\rho}(0)}(|\nabla u|^{2}+|\nabla v|^{2})dx+o_{n}(1)\\
&\leq\int_{B_{\rho}(0)}(|\nabla \widetilde{u}_{n}|^{2}+|\nabla \widetilde{v}_{n}|^{2})dx+o_{n}(1).
\endaligned
\end{equation}
Since $\widehat{\varphi}_{n}\rightarrow0$ and $\widehat{\psi}_{n}\rightarrow0$
in $H^{1}(B_{3}(0)\backslash B_{\rho}(0))$, the scale invariance implies that
\begin{equation}\label{eq4.12}
\aligned
o_{n}(1)
&=\langle J'_{\infty}(\overline{u}_{n},\overline{v}_{n}),(\varphi_{n},\psi_{n})\rangle
=\langle J'_{\infty}(\widetilde{u}_{n},\widetilde{v}_{n}),(\widetilde{\varphi}_{n},\widetilde{\psi}_{n})\rangle\\
&=\int_{B_{\rho}(0)}(\nabla \widetilde{u}_{n}\nabla (\widetilde{u}_{n}-u)
+\nabla \widetilde{v}_{n}\nabla (\widetilde{v}_{n}-v))dx\\
&\quad\,-\alpha_1\int_{B_{\rho}(0)}\int_{\R^N}
\frac{|\widetilde{u}_{n}^{+}(x)|^{2}\widetilde{u}_{n}^{+}(y)(\widetilde{u}_{n}-u)(y)}{|x-y|^{4}}dxdy
-\beta\int_{B_{\rho}(0)}\int_{\R^N}
\frac{|\widetilde{v}_{n}^{+}(x)|^{2}\widetilde{u}_{n}^{+}(y)(\widetilde{u}_{n}-u)(y)}{|x-y|^{4}}dxdy\\
&\quad\,-\alpha_2\int_{B_{\rho}(0)}\int_{\R^N}
\frac{|\widetilde{v}_{n}^{+}(x)|^{2}\widetilde{v}_{n}^{+}(y)(\widetilde{v}_{n}-v)(y)}{|x-y|^{4}}dxdy
-\beta\int_{B_{\rho}(0)}\int_{\R^N}
\frac{|\widetilde{u}_{n}^{+}(x)|^{2}\widetilde{v}_{n}^{+}(y)(\widetilde{v}_{n}-v)(y)}{|x-y|^{4}}dxdy\\
&\quad\,+o_{n}(1).
\endaligned
\end{equation}
Since $\{\widetilde{u}_{n}^{2}\}$ is bounded in $L^{\frac{N}{N-2}}(\R^N)$
and $\widetilde{u}_{n}\rightarrow u$ almost everywhere in $\R^N$, we
have $|\widetilde{u}_{n}^{+}|^{2}\rightharpoonup |u^{+}|^{2}$ in
$L^{\frac{N}{N-2}}(\R^N)$. By the Hardy-Littlewood-Sobolev inequality,
the Riesz potential defines a linear continuous map from
$L^{\frac{N}{N-2}}(\R^N)$ to $L^{\frac{N}{2}}(\R^N)$ and then
$$
\int_{\R^N}
\frac{|\widetilde{u}_{n}^{+}(x)|^{2}}{|x-y|^{4}}dx\rightharpoonup
\int_{\R^N}\frac{|u^{+}(x)|^{2}}{|x-y|^{4}}dx\ \ \mbox{in}\ L^{\frac{N}{2}}(\R^N).
$$
Combining this with $\widetilde{u}^{+}_{n}\rightharpoonup u^{+}$ in
$L^{\frac{2N}{N-2}}(\R^N)$ leads to
$$
\widetilde{u}_{n}^{+}(y)
\int_{\R^N}\frac{|\widetilde{u}_{n}^{+}(x)|^{2}}{|x-y|^{4}}dx \rightharpoonup
u^{+}(y)\int_{\R^N}\frac{|u^{+}(x)|^{2}}{|x-y|^{4}}dx\ \ \mbox{in} \ L^{\frac{2N}{N+2}}(\R^N),
$$
which implies that
\begin{align*}
\lim_{n\to\infty}\int_{B_{\rho}(0)}\int_{\R^N}
\frac{|\widetilde{u}_{n}^{+}(x)|^{2}\widetilde{u}_{n}^{+}(y)u(y)}{|x-y|^{4}}dxdy
&=\int_{B_{\rho}(0)}\int_{\R^N}
\frac{|u^{+}(x)|^{2}u^{+}(y)u(y)}{|x-y|^{4}}dxdy\\
&=\int_{B_{\rho}(0)}\int_{\R^N}
\frac{|u^{+}(x)|^{2}|u^{+}(y)|^{2}}{|x-y|^{4}}dxdy.
\end{align*}
Then
\begin{equation}\label{eq4.13}
\aligned
&\quad\,\int_{B_{\rho}(0)}\int_{\R^N}
\frac{|\widetilde{u}_{n}^{+}(x)|^{2}\widetilde{u}_{n}^{+}(y)(\widetilde{u}_{n}-u)(y)}{|x-y|^{4}}dxdy\\
&=\int_{B_{\rho}(0)}\int_{\R^N}
\frac{|\widetilde{u}_{n}^{+}(x)|^{2}|\widetilde{u}_{n}^{+}(y)|^{2}}{|x-y|^{4}}dxdy-
\int_{B_{\rho}(0)}\int_{\R^N}
\frac{|u^{+}(x)|^{2}|u^{+}(y)|^{2}}{|x-y|^{4}}dxdy+o_{n}(1)\\
&=\int_{B_{\rho}(0)}\int_{\R^N}
\frac{|(\widetilde{u}_{n}-u)^{+}(x)|^{2}|(\widetilde{u}_{n}-u)^{+}(y)|^{2}}{|x-y|^{4}}dxdy+o_{n}(1).
\endaligned
\end{equation}
Similarly, we also have
\begin{equation}\label{eq4.14}
\aligned
\int_{B_{\rho}(0)}\int_{\R^N}
\frac{|\widetilde{v}_{n}^{+}(x)|^{2}\widetilde{u}_{n}^{+}(y)(\widetilde{u}_{n}-u)(y)}{|x-y|^{4}}dxdy
=\int_{B_{\rho}(0)}\int_{\R^N}
\frac{|(\widetilde{v}_{n}-v)^{+}(x)|^{2}|(\widetilde{u}_{n}-u)^{+}(y)|^{2}}{|x-y|^{4}}dxdy
+o_{n}(1),
\endaligned
\end{equation}
\begin{equation}\label{eq4.15}
\aligned
\int_{B_{\rho}(0)}\int_{\R^N}
\frac{|\widetilde{v}_{n}^{+}(x)|^{2}\widetilde{v}_{n}^{+}(y)(\widetilde{v}_{n}-v)(y)}{|x-y|^{4}}dxdy
=\int_{B_{\rho}(0)}\int_{\R^N}
\frac{|(\widetilde{v}_{n}-v)^{+}(x)|^{2}|(\widetilde{v}_{n}-v)^{+}(y)|^{2}}{|x-y|^{4}}dxdy
+o_{n}(1)
\endaligned
\end{equation}
and
\begin{equation}\label{eq4.16}
\aligned
\int_{B_{\rho}(0)}\int_{\R^N}
\frac{|\widetilde{u}_{n}^{+}(x)|^{2}\widetilde{v}_{n}^{+}(y)(\widetilde{v}_{n}-v)(y)}{|x-y|^{4}}dxdy
=\int_{B_{\rho}(0)}\int_{\R^N}
\frac{|(\widetilde{u}_{n}-u)^{+}(x)|^{2}|(\widetilde{v}_{n}-v)^{+}(y)|^{2}}{|x-y|^{4}}dxdy+o_{n}(1).
\endaligned
\end{equation}
Substituting \eqref{eq4.13}$-$\eqref{eq4.16} into \eqref{eq4.12} and using
$(\widetilde{u}_{n},\widetilde{v}_{n})\rightharpoonup(u, v)$ in $H$,
we obtain
$$
\aligned
o_{n}(1)
&=\int_{B_{\rho}(0)}(|\nabla (\widetilde{u}_{n}-u)|^{2}+|\nabla (\widetilde{v}_{n}-v)|^{2})dx\\
&\quad\,-\alpha_1\int_{B_{\rho}(0)}\int_{\R^N}
\frac{|(\widetilde{u}_{n}-u)^{+}(x)|^{2}|(\widetilde{u}_{n}-u)^{+}(y)|^{2}}{|x-y|^{4}}dxdy
-\beta\int_{B_{\rho}(0)}\int_{\R^N}
\frac{|(\widetilde{v}_{n}-v)^{+}(x)|^{2}|(\widetilde{u}_{n}-u)^{+}(y)|^{2}}{|x-y|^{4}}dxdy\\
&\quad\,-\alpha_2\int_{B_{\rho}(0)}\int_{\R^N}
\frac{|(\widetilde{v}_{n}-v)^{+}(x)|^{2}|(\widetilde{v}_{n}-v)^{+}(y)|^{2}}{|x-y|^{4}}dxdy
-\beta\int_{B_{\rho}(0)}\int_{\R^N}
\frac{|(\widetilde{u}_{n}-u)^{+}(x)|^{2}|(\widetilde{v}_{n}-v)^{+}(y)|^{2}}{|x-y|^{4}}dxdy.
\endaligned$$
Using $\widehat{\varphi}_{n}\rightarrow0$ and $\widehat{\psi}_{n}\rightarrow0$
in $H^{1}(B_{3}(0)\backslash B_{\rho}(0))$ and the scale invariance again
\begin{equation}\label{eq4.17}
\aligned
o_{n}(1)
=&\int_{\R^N}(|\nabla \widetilde{\varphi}_{n}|^{2}+|\nabla \widetilde{\psi}_{n}|^{2})dx
-\alpha_1\int_{\R^N}\int_{\R^N}
\frac{|\widetilde{\varphi}_{n}^{+}(x)|^{2}|\widetilde{\varphi}_{n}^{+}(y)|^{2}}{|x-y|^{4}}dxdy\\
&-\alpha_2\int_{\R^N}\int_{\R^N}
\frac{|\widetilde{\psi}_{n}^{+}(x)|^{2}|\widetilde{\psi}_{n}^{+}(y)|^{2}}{|x-y|^{4}}dxdy
-2\beta\int_{\R^N}\int_{\R^N}
\frac{|\widetilde{\varphi}_{n}^{+}(x)|^{2}|\widetilde{\psi}_{n}^{+}(y)|^{2}}{|x-y|^{4}}dxdy\\
=&\int_{\R^N}(|\nabla \varphi_{n}|^{2}+|\nabla \psi_{n}|^{2})dx
-\alpha_1\int_{\R^N}\int_{\R^N}
\frac{|\varphi_{n}^{+}(x)|^{2}|\varphi_{n}^{+}(y)|^{2}}{|x-y|^{4}}dxdy\\
&-\alpha_2\int_{\R^N}\int_{\R^N}
\frac{|\psi_{n}^{+}(x)|^{2}|\psi_{n}^{+}(y)|^{2}}{|x-y|^{4}}dxdy
-2\beta\int_{\R^N}\int_{\R^N}
\frac{|\varphi_{n}^{+}(x)|^{2}|\psi_{n}^{+}(y)|^{2}}{|x-y|^{4}}dxdy.
\endaligned
\end{equation}
If $(\varphi_{n}^+,\psi_{n}^+)\neq(0,0)$, we define $t_{n} > 0$ by
$$
t_{n}^{2}=\frac{\displaystyle\|(\varphi_{n},\psi_{n})\|^{2}}
{\displaystyle\int_{\R^N}\int_{\R^N}
\frac{\alpha_1|\varphi_{n}^{+}(x)|^{2}|\varphi_{n}^{+}(y)|^{2}
+\alpha_2|\psi_{n}^{+}(x)|^{2}|\psi_{n}^{+}(y)|^{2}
+2\beta|\varphi_{n}^{+}(x)|^{2}|\psi_{n}^{+}(y)|^{2}}{|x-y|^{4}}dxdy}.
$$
Then $(t_{n}\varphi_{n},t_{n}\psi_{n})\in \mathcal{N}_{\infty}$ and
$$
c_{\infty}\leq J_{\infty}(t_{n}\varphi_{n},t\psi_{n})
=\frac{1}{4}\frac{\displaystyle\|(\varphi_{n},\psi_{n})\|^{4}}
{\displaystyle\int_{\R^N}\int_{\R^N}
\frac{\alpha_1|\varphi_{n}^{+}(x)|^{2}|\varphi_{n}^{+}(y)|^{2}
+\alpha_2|\psi_{n}^{+}(x)|^{2}|\psi_{n}^{+}(y)|^{2}
+2\beta|\varphi_{n}^{+}(x)|^{2}|\psi_{n}^{+}(y)|^{2}}{|x-y|^{4}}dxdy},
$$
which indicates that
$$
\int_{\R^N}\int_{\R^N}
\frac{\alpha_1|\varphi_{n}^{+}(x)|^{2}|\varphi_{n}^{+}(y)|^{2}
+\alpha_2|\psi_{n}^{+}(x)|^{2}|\psi_{n}^{+}(y)|^{2}
+2\beta|\varphi_{n}^{+}(x)|^{2}|\psi_{n}^{+}(y)|^{2}}{|x-y|^{4}}dxdy
\leq\frac{1}{4c_{\infty}}\|(\varphi_{n},\psi_{n})\|^{4}.
$$
Note that the above inequality also holds if $(\varphi_{n}^+,\psi_{n}^+)=(0,0)$.
Combining this with \eqref{eq4.17} and \eqref{eq4.11} yields
$$
\aligned
o_{n}(1)
&\geq\Big(1-\frac{1}{4c_{\infty}}\int_{\R^N}(|\nabla \varphi_{n}|^{2}+|\nabla \psi_{n}|^{2})dx\Big)
\int_{\R^N}(|\nabla \varphi_{n}|^{2}+|\nabla \psi_{n}|^{2})dx\\
&\geq\Big(1-\frac{1}{4c_{\infty}}\int_{B_{\rho}(0)}
(|\nabla \widetilde{u}_{n}|^{2}+|\nabla \widetilde{v}_{n}|^{2})dx\Big)
\int_{\R^N}(|\nabla \varphi_{n}|^{2}+|\nabla \psi_{n}|^{2})dx+o_{n}(1).
\endaligned$$
Therefore, we have
$$
\lim_{n\rightarrow\infty}\int_{B_{\rho}(0)}
(|\nabla (\widetilde{u}_{n}-u)|^{2}+|\nabla (\widetilde{v}_{n}-v)|^{2})dx
=\lim_{n\rightarrow\infty}\int_{\R^N}(|\nabla \varphi_{n}|^{2}+|\nabla \psi_{n}|^{2})dx=0,
$$
since the definition of $L$ implies
$$
\int_{B_{\rho}(0)}(|\nabla \widetilde{u}_{n}|^{2}+|\nabla \widetilde{v}_{n}|^{2})dx
\leq L\int_{B_{1}(0)}(|\nabla \widetilde{u}_{n}|^{2}+|\nabla \widetilde{v}_{n}|^{2})dx
=2c_{\infty}.
$$
Then we see from \eqref{eq4.10} that $(u, v)\neq (0, 0)$ and conclude the
proof of the claim.

Set $(u_n^1, v_n^1):=(\overline{u}_{n}, \overline{v}_{n})$, $(u^1, v^1):=(u,v)$,
$r_n^1:=r_n$ and $z_n^1:=z_n$. Doing iteration, we obtain sequences
$\{r_{n}^{j}\}$ and $\{z_{n}^{j}\}$ such that
$(u_{n}^{j},v_{n}^{j}):=(u_{n}^{j-1}-u^{j-1},v_{n}^{j-1}-v^{j-1})_{r_{n}^{j},z_{n}^{j}}
\rightharpoonup (u^{j},v^{j})$ in $H$, where $(u^{j},v^{j})$ are nonzero
solutions of \eqref{eq2.3}. Moreover, we see from \eqref{eq4.7} and
\eqref{eq4.8} that, by induction,
$$
\|(u_{n}^{j},v_{n}^{j})\|^{2}=
\|(u_{n},v_{n})\|^{2}-\sum_{i=0}^{j-1}\|(u^{i},v^{i})\|^{2}+o_{n}(1)
$$
and
$$
J_{\infty}(u_{n}^{j},v_{n}^{j})
=J(u_{n},v_{n})-J(u^{0},v^{0})-\sum_{i=1}^{j-1}J_{\infty}(u^{i},v^{i})+o_{n}(1).
$$
The iterating process must terminate in finite steps, because, for any
nonzero solution $(u, v)$ of \eqref{eq2.3}, there holds
$J_{\infty}(u, v)\geq c_{\infty}>0$. Moreover, the last Palais-Smale
sequence for $J_{\infty}$ must converge to $(0, 0)$ strongly in $H$.
This finishes the proof.
\end{proof}

\begin{cor}\label{cor4.3}
Let $\{(u_{n},v_{n})\}\subset \mathcal{N}$ be a $(PS)_{d}$ sequence
for the constrained functional $J|_\cn$ at the level
$d\in (c_{\infty}, \min\{S_{H,L}^{2}/4\alpha_1, S_{H,L}^{2}/4\alpha_2, 2c_{\infty}\})$,
then $\{(u_{n},v_n)\}$ is relatively compact in $H$.
\end{cor}

\begin{proof}[\bf Proof.]
It is easy to see that $\{(u_{n},v_{n})\}\subset \mathcal{N}$ is a $(PS)_{d}$
sequence for the functional $J$. According to Lemma \ref{lem4.2}, there
exist a number $k\in \N$, a solution $(u^{0},v^{0})$ of \eqref{eq2.1}
and nonzero solutions $(u^1, v^1),\cdots, (u^k, v^k)$ of \eqref{eq2.3}
such that, up to a subsequence,
$$
\lim_{n\to\infty}\|(u_{n},v_{n})\|^{2}=\sum_{j=0}^{k}\|(u^{j},v^{j})\|^{2}
$$
and
$$
\lim_{n\to\infty}J(u_{n},v_{n})=J(u^{0},v^{0})+\sum_{j=1}^{k}J_{\infty}(u^{j},v^{j}).
$$
We first claim that $(u^0, v^0) \neq (0, 0)$. If not, we see from $d<2c_\infty$
that $k=1$. By Corollary \ref{cor3.2} and the uniqueness of positive
solutions for the Choquard equation $-\Delta u=\alpha_{i}(|x|^{-4}\ast u^{2})u$
in $\R^N$, $(u^1, v^1)$ must be, up to translation and dilation, one of the
three solutions $(\sqrt{k_{0}}U_{1,0},\sqrt{l_{0}}U_{1,0})$,
$(\frac{1}{\sqrt{\alpha_1}}U_{1,0},0)$, and $(0,\frac{1}{\sqrt{\alpha_2}}U_{1,0})$.
Then either $d = c_{\infty}$ or $d =S_{H,L}^{2}/4\alpha_1$ or
$d =S_{H,L}^{2}/4\alpha_2$, which contradicts the assumption. since
$(u^0, v^0) \neq (0, 0)$, using $d<2c_\infty$ again and Lemma \ref{lem2.4}
leads to $k=0$. Therefore, we conclude that $\{(u_{n},v_{n})\}$ is relatively
compact in $H$.
\end{proof}

\section{Existence of a positive solution}
For $(u,v)\in H$, set
$$
\|(u,v)\|_{NL}:=\Big(\int_{\R^N}\int_{\R^N}
\frac{\alpha_1|u^{+}(x)|^{2}|u^{+}(y)|^{2}+\alpha_2|v^{+}(x)|^{2}|v^{+}(y)|^{2}
+2\beta|u^{+}(x)|^{2}|v^{+}(y)|^{2}}{|x-y|^{4}}dxdy\Big)^{\frac14}.
$$
Following the idea in \cite{CeMo}, we introduce a barycenter map $\beta:H\rightarrow\R^{N}$ defined
as
$$
\beta(u,v)=\frac{1}{\|(u,v)\|_{NL}^4}\int_{\R^N}\frac{x}{1+|x|}\int_{\R^N}
\frac{\alpha_1|u^{+}(x)|^{2}|u^{+}(y)|^{2}+\alpha_2|v^{+}(x)|^{2}|v^{+}(y)|^{2}
+2\beta|u^{+}(x)|^{2}|v^{+}(y)|^{2}}{|x-y|^{4}}dy\,dx.
$$
We also define a functional
\begin{multline*}
\gamma(u,v)=\\
\frac{1}{\|(u,v)\|_{NL}^4}\int_{\R^N}\Big|\frac{x}{1+|x|}-\beta(u,v)\Big|
\int_{\R^N}\frac{\alpha_1|u^{+}(x)|^{2}|u^{+}(y)|^{2}+\alpha_2|v^{+}(x)|^{2}|v^{+}(y)|^{2}
+2\beta|u^{+}(x)|^{2}|v^{+}(y)|^{2}}{|x-y|^{4}}dy\,dx
\end{multline*}
to estimate the concentration of $(u,v)$ around its barycenter. Denote
$$
\mathcal{M}:=\Big\{(u,v)\in\mathcal{N}:\beta(u,v)=0,\gamma(u,v)=\frac{1}{2}\Big\}
$$
and consider the infimum
$$
c^{\star}=\inf_{(u,v)\in\mathcal{M}}J(u,v).
$$

\begin{lem}\label{lem5.1}
If $\beta>\max\{\alpha_1,\alpha_2\}$ and $V_1,\,V_2\in L^{\frac{N}{2}}(\R^{N})$
are nonnegative functions such that
$$
|V_1|_{\frac{N}{2}}+|V_2|_{\frac{N}{2}}>0,
$$
then $c^{\star} >c_{\infty}$.
\end{lem}

\begin{proof}[\bf Proof.]
We first see from $\cm\subset\cn$ and Lemma \ref{lem2.4} that
$c^{\star} \geq c=c_{\infty}$. Assume to the contrary that $c^{\star}=c_{\infty}$.
Then, by Ekeland's variational principle, there exists a sequence
$\{(u_{n},v_{n})\}\subset \cn$ such that, as $n\to\infty$,
\begin{equation}\label{eq5.1}
\beta(u_{n},v_{n})\to 0,\ \ \gamma(u_{n},v_{n})\to\frac{1}{2}
\end{equation}
and
$$
J(u_{n},v_{n})\to c_{\infty},\ \ (J|_\cn)'(u_n, v_n)\to 0.
$$
By Lemmas \ref{lem2.4} and \ref{lem4.2}, there exist
$\delta_{n}> 0$, $z_{n}\in\R^N$ and $(\varphi_{n},\psi_{n})\in H$
such that
$$
(u_{n},v_{n})=(\sqrt{k_{0}}U_{\delta_{n},z_{n}},\sqrt{l_{0}}U_{\delta_{n},z_{n}})
+(\varphi_{n},\psi_{n}),
$$
where $(\varphi_{n},\psi_{n})\rightarrow(0, 0)$ in $H$. Then we have
\begin{equation}\label{eq5.2}
\aligned
\frac{1}{2}=
&\lim_{n\rightarrow\infty}\gamma(u_{n},v_{n})\\
=&\lim_{n\rightarrow\infty}\frac{1}{\|(u_n,v_n)\|_{NL}^4}\int_{\R^N}
\frac{|x|}{1+|x|}\int_{\R^N}\frac{\alpha_1|u_{n}^{+}(x)|^{2}|u_{n}^{+}(y)|^{2}
+\alpha_2|v_{n}^{+}(x)|^{2}|v_{n}^{+}(y)|^{2}
+2\beta|u_{n}^{+}(x)|^{2}|v_{n}^{+}(y)|^{2}}{|x-y|^{4}}dy\,dx\\
=&\lim_{n\rightarrow\infty}\frac{1}{S_{H,L}^{2}}\int_{\R^N}
\frac{|x|}{1+|x|}
\int_{\R^N}\frac{U^2_{\delta_{n},z_{n}}(x) U^2_{\delta_{n},z_{n}}(y)}{|x-y|^{4}}dy\,dx
\endaligned
\end{equation}
and
\begin{equation}\label{eq5.3}
0=\lim_{n\rightarrow\infty}\beta(u_{n},v_{n})
=\lim_{n\rightarrow\infty}\frac{1}{S_{H,L}^{2}}\int_{\R^N}
\frac{x}{1+|x|}
\int_{\R^N}\frac{U^2_{\delta_{n},z_{n}}(x) U^2_{\delta_{n},z_{n}}(y)}{|x-y|^{4}}dy\,dx.
\end{equation}
We divide into the following four cases. In each case, we will come to a
contradiction and finish the proof.

{\bf Case 1.} Up to a subsequence, there holds
$\lim_{n\rightarrow\infty}\delta_{n}=+\infty$.

In this case, we have
$$
\lim_{n\rightarrow\infty}\int_{B_{r}(0)} \int_{\R^N}
\frac{U^2_{\delta_{n},z_{n}}(x) U^2_{\delta_{n},z_{n}}(y)}{|x-y|^{4}}dy\,dx=0
$$
for any $r > 0$. Combining this with \eqref{eq5.2} leads to
$$
\frac{1}{2}
=\lim_{n\rightarrow\infty}\frac{1}{S_{H,L}^{2}}\int_{\R^N\backslash B_{r}(0)}
\frac{|x|}{1+|x|}
\int_{\R^N}\frac{U^2_{\delta_{n},z_{n}}(x) U^2_{\delta_{n},z_{n}}(y)}{|x-y|^{4}}dy\,dx
\geq\frac{r}{1+r},
$$
which is impossible when $r>1$.

{\bf Case 2.} Up to a subsequence, there hold
$\lim_{n\rightarrow\infty}\delta_{n}=\delta>0$ and $\lim_{n\to\infty}z_n=z$.

In this case, one can prove that
$$
(\sqrt{k_{0}}U_{\delta_{n},z_{n}},\sqrt{l_{0}}U_{\delta_{n},z_{n}})
\to (\sqrt{k_{0}}U_{\delta,z},\sqrt{l_{0}}U_{\delta,z})\ \ \text{in}\ H
$$
and then $(u_{n},v_{n})\to (\sqrt{k_{0}}U_{\delta,z},\sqrt{l_{0}}U_{\delta,z})$
in $H$. We come to a contradiction as shown by
\begin{align*}
c_\infty
&=\lim_{n\to\infty}J(u_n, v_n)\\
&=\frac14\lim_{n\to\infty}\int_{\R^N}(|\nabla u_n|^2+|\nabla v_n|^2+V_1(x)u_n^2+V_2(x)v_n^2)dx\\
&=\frac{k_0+l_0}{4}\int_{\R^N}|\nabla U_{\delta,z}|^2dx
+\frac14\int_{\R^4}(k_0V_1(x)U_{\delta,z}^2+l_0V_2(x)U_{\delta,z}^2)\\
&>\frac{k_0+l_0}{4}S_{H,L}^2\\
&=c_\infty.
\end{align*}

{\bf Case 3.} Up to a subsequence, there holds
$\lim_{n\rightarrow\infty}\delta_{n}=\delta>0$ and $\lim_{n\to\infty}|z_n|=+\infty$.

In this case, we have
\begin{align*}
&\quad\,\frac{1}{S_{H,L}^{2}}\int_{\R^N}
\frac{|x|}{1+|x|}
\int_{\R^N}\frac{U^2_{\delta_{n},z_{n}}(x) U^2_{\delta_{n},z_{n}}(y)}{|x-y|^{4}}dy\,dx\\
&=1-\frac{1}{S_{H,L}^{2}}\int_{\R^N}
\frac{1}{1+|x|}
\int_{\R^N}\frac{U^2_{\delta_{n},z_{n}}(x) U^2_{\delta_{n},z_{n}}(y)}{|x-y|^{4}}dy\,dx\\
&=1-o_n(1),
\end{align*}
which contradicts with \eqref{eq5.2}.

{\bf Case 4.} Up to a subsequence, there holds
$\lim_{n\rightarrow\infty}\delta_{n}=0$.

In this case, we have
$$
\lim_{n\rightarrow\infty}\int_{\R^N\backslash B_{r}(0)}
\int_{\R^N}\frac{U^2_{\delta_{n},0}(x) U^2_{\delta_{n},0}(y)}{|x-y|^{4}}dy\,dx
=0
$$
for any $r > 0$. Combining this with \eqref{eq5.1}, \eqref{eq5.3} and the
inequality $|\frac{z_{n}}{1+|z_{n}|}-\frac{x}{1+|x|}|\leq r$ for $x\in B_{r}(z_{n})$
yields
$$
\aligned
\frac{|z_{n}|}{1+|z_{n}|}
&=\Big|\frac{z_{n}}{1+|z_{n}|}-\beta(u_{n},v_{n})\Big|+o_{n}(1)\\
&\leq\frac{1}{S_{H,L}^{2}}\int_{\R^N}
\Big|\frac{z_{n}}{1+|z_{n}|}-\frac{x}{1+|x|}\Big|\int_{\R^N}
\frac{U^2_{\delta_{n},z_{n}}(x) U^2_{\delta_{n},z_{n}}(y)}{|x-y|^{4}}dy\,dx+o_{n}(1)\\
&=\frac{1}{S_{H,L}^{2}}\int_{B_{r}(z_n)}
\Big|\frac{z_{n}}{1+|z_{n}|}-\frac{x}{1+|x|}\Big|\int_{\R^N}
\frac{U^2_{\delta_{n},z_{n}}(x) U^2_{\delta_{n},z_{n}}(y)}{|x-y|^{4}}dy\,dx+o_{n}(1)\\
&\leq r+o_{n}(1),
\endaligned
$$
which implies that $|z_n|\rightarrow0$ as $n\rightarrow\infty$, since $r>0$
is arbitrary. Then we have
$$
\aligned
&\quad\,\frac{1}{S_{H,L}^{2}}\int_{\R^N}\frac{|x|}{1+|x|}
\int_{\R^N}\frac{U^2_{\delta_{n},z_{n}}(x) U^2_{\delta_{n},z_{n}}(y)}{|x-y|^{4}}dy\,dx\\
&=\frac{1}{S_{H,L}^{2}}\int_{B_{r}(z_n)}\frac{|x|}{1+|x|}
\int_{\R^N}\frac{U^2_{\delta_{n},z_n}(x) U^2_{\delta_{n},z_n}(y)}{|x-y|^{4}}dy\,dx+o_n(1)\\
&\leq r+o_n(1)
\endaligned
$$
for any $r>0$. This contradicts with \eqref{eq5.2} again when $r<\frac12$.
\end{proof}

Let
$\eta\in \mathcal{C}_{0}^{\infty}(B_{1}(0))$ be a radially decreasing
function such that  $\eta\equiv1$ on $B_{\rho}(0)$ for some
$0<\rho<1$ and define $u_{\varepsilon}(x) =  \eta(x)U_{\varepsilon,0}(x)$
for $\varepsilon> 0$. By \cite[Section 3]{GY}, we have
\begin{equation}\label{eq5.6}
\int_{\R^N}|\nabla u_{\varepsilon}|^{2}dx
=S_{H,L}^2+O(\varepsilon^{N-2}),
\end{equation}
\begin{equation}\label{eq5.7}
\int_{\R^N}\int_{\R^N}\frac{u^2_{\varepsilon}(x) u^2_{\varepsilon}(y)}
{|x-y|^{4}}dxdy
\geq S_{H,L}^2-O(\varepsilon^{N-2})
\end{equation}
and
\begin{equation}\label{eq5.8}
\int_{\R^N}\int_{\R^N}\frac{u^2_{\varepsilon}(x) u^2_{\varepsilon}(y)}
{|x-y|^{4}}dxdy
\leq S_{H,L}^2+O(\varepsilon^{2N-4}).
\end{equation}
Set
$$
t_{\varepsilon}=\bigg(\frac{\displaystyle\int_{\R^N}|\nabla u_{\varepsilon}|^{2}dx}
{\displaystyle\int_{\R^N}\int_{\R^N}\frac{u^2_{\varepsilon}(x) u^2_{\varepsilon}(y)}
{|x-y|^{4}}dxdy}\bigg)^{\frac12}.
$$
Then
$$
\aligned
\|t_{\varepsilon}u_{\varepsilon}\|^{2}
&=\int_{\R^N}\int_{\R^N}
\frac{|t_{\varepsilon}u_{\varepsilon}(x)|^{2}|t_{\varepsilon}u_{\varepsilon}(y)|^{2}}
{|x-y|^{4}}dxdy\\
&=\frac{\|u_{\varepsilon}\|^{4}}{\displaystyle\int_{\R^N}\int_{\R^N}
\frac{u^2_{\varepsilon}(x) u^2_{\varepsilon}(y)}{|x-y|^{4}}dxdy}\\
&\geq\frac{\big[S_{H,L}^2+O(\varepsilon^{N-2})\big]^{2}}
{S_{H,L}^2+O(\varepsilon^{2N-4})}\\
&=S_{H,L}^{2}+O(\varepsilon^{N-2})\\
&>S_{H,L}^{2}
\endaligned$$
for $\varepsilon> 0$ sufficiently small. By \eqref{eq5.6}$-$\eqref{eq5.8}, it is
easy to see that
$$
\lim_{\varepsilon\rightarrow0}\|t_{\varepsilon}u_{\varepsilon}\|^{2}
=\lim_{\varepsilon\rightarrow0}\int_{\R^N}\int_{\R^N}
\frac{|t_{\varepsilon}u_{\varepsilon}(x)|^{2}|t_{\varepsilon}u_{\varepsilon}(y)|^{2}}
{|x-y|^{4}}dxdy
=S_{H,L}^{2}.
$$
Take $w = t_{\varepsilon}u_{\varepsilon}$ with $\varepsilon>0$
small enough, we can see that
the nonnegative radial function $w\in \mathcal{C}_{0}^{\infty}(\R^N)$
satisfies the following properties:
$\mbox{supp}\,w \subset B_1(0)$, $w$ is non-increasing with respect
to $r =|x|$,
$$
\|w\|^{2}=\int_{\R^N}\int_{\R^N}\frac{w^2(x) w^2(y)}{|x-y|^{4}}dxdy>S_{H,L}^{2},
$$
\begin{equation}\label{eq5.4}
\frac{k_{0}+l_{0}}{4}\|w\|^{2}<c^{\star}
\end{equation}
and
\begin{equation}\label{eq5.5}
\aligned
\frac{k_{0}+l_{0}}{4}\|w\|^{2}S_{H,L}^{-2}\Big(S_{H,L}
+\frac{\beta-\alpha_2}{2\beta-\alpha_1-\alpha_2}C(N, 4)^{-\frac12}|V_1|_{\frac{N}{2}}
+&\frac{\beta-\alpha_1}{2\beta-\alpha_1-\alpha_2}C(N, 4)^{-\frac12}|V_2|_{\frac{N}{2}}\Big)^{2}\\
&\qquad\quad<\min\Big\{\frac{S_{H,L}^{2}}{4\alpha_1},
\frac{S_{H,L}^{2}}{4\alpha_2},2c_{\infty}\Big\}
\endaligned
\end{equation}
which is equivalent to the second inequality in \eqref{eq1.10}.

The proof of the following Lemma is similar to Lemma 3.6 in \cite{CeMo}, Lemma 4.2 in \cite{LL}, we only state the main results here.
\begin{lem}\label{lem5.2}
Denote $w_{\delta,z}=\delta^{-\frac{N-2}{2}}w(\frac{x-z}{\delta})$ for
$\delta> 0$ and $z\in\R^N$. If $a\in L^{\frac{N}{2}}(\R^N)$, then
$$
\lim_{\delta\rightarrow0^{+}}\int_{\R^N}a(x)w_{\delta,z}^{2}dx
=\lim_{\delta\rightarrow+\infty}\int_{\R^N}a(x)w_{\delta,z}^{2}dx
=0
$$
uniformly for $z\in\R^N$ and
$$
\lim_{|z|\rightarrow+\infty}\int_{\R^N}a(x)w_{\delta,z}^{2}dx=0
$$
uniformly for $\delta> 0$.
\end{lem}

\begin{lem}\label{lem5.3}
Denote the inner product in $\R^N$ by $\langle\cdot,\cdot\rangle_{\R^N}$
and let $r> 0$ be a fixed number. Then

\vskip 1mm
\noindent $(1)$ $\langle\beta(\sqrt{k_{0}}w_{\delta,z},\sqrt{l_{0}}w_{\delta,z}),z\rangle_{\R^N}>0$
for any $\delta> 0$ and $z\in\R^N\backslash\{0\}$;

\vskip 1mm
\noindent  $(2)$ $\lim_{\delta\rightarrow0^{+}}\gamma(\sqrt{k_{0}}w_{\delta,z},\sqrt{l_{0}}w_{\delta,z})=0$
uniformly for $z\in\R^N$;

\vskip 1mm
\noindent $(3)$ $\lim_{\delta\rightarrow+\infty}\gamma(\sqrt{k_{0}}w_{\delta,z},\sqrt{l_{0}}w_{\delta,z})=1$
uniformly for $z\in B_{r}(0)$.
\end{lem}

\begin{proof}[\bf Proof.]
(1) Let $\delta> 0$ and $z\in\R^N\backslash\{0\}$. For any $x\in\R^N$ with
$\langle x,z\rangle_{\R^N}>0$, there holds $|-x-z|>|x-z|$. Then from the
properties of $w$ we see that $w_{\delta,z}(x)\geq w_{\delta,z}(-x)$ for any
$x\in\R^N$ with  $\langle x,z\rangle_{\R^N}>0$ and
meas\,$\{x\in\R^N|\langle x,z\rangle_{\R^N}>0,w_{\delta,z}(x)> w_{\delta,z}(-x)\}>0$.
Thus we have
$$
\aligned
&\quad\,\int_{\R^{N}}\frac{\langle x,z\rangle_{\R^N}}{1+|x|}
\int_{\R^N}\frac{w^2_{\delta,z}(x) w^2_{\delta,z}(y)}{|x-y|^{4}}dy\,dx\\
&=\int_{\{x\in\R^{N}|\langle x,z\rangle_{\R^N}>0\}}\frac{\langle x,z\rangle_{\R^N}}{1+|x|}
\int_{\R^N}\frac{w^2_{\delta,z}(x) w^2_{\delta,z}(y)}{|x-y|^{4}}dy\,dx\\
&\quad\,+\int_{\{x\in\R^{N}|\langle x,z\rangle_{\R^N}<0\}}\frac{\langle x,z\rangle_{\R^N}}{1+|x|}
\int_{\R^N}\frac{w^2_{\delta,z}(x) w^2_{\delta,z}(y)}{|x-y|^{4}}dy\,dx\\
&=\int_{\{x\in\R^{N}|\langle x,z\rangle_{\R^N}>0\}}\frac{\langle x,z\rangle_{\R^N}}{1+|x|}
\int_{\R^N}\frac{(w^2_{\delta,z}(x)-w^2_{\delta,z}(-x)) w^2_{\delta,z}(y)}{|x-y|^{4}}dy\,dx\\
&>0,
\endaligned
$$
which indicates that
$$
\langle\beta(\sqrt{k_{0}}w_{\delta,z},\sqrt{l_{0}}w_{\delta,z}),z\rangle_{\R^N}
=\frac{1}{\|w\|^{2}}\int_{\R^N}\frac{\langle x,z\rangle_{\R^N}}{1+|x|}
\int_{\R^N}\frac{w^2_{\delta,z}(x) w^2_{\delta,z}(y)}{|x-y|^{4}}dy\,dx
>0
$$
for any $\delta> 0$ and $z\in\R^N\backslash\{0\}$.

(2) For any $\delta> 0$ and $z\in\R^N$, we have
$$
\Big|\frac{z}{1+|z|}-\beta(\sqrt{k_{0}}w_{\delta,z},\sqrt{l_{0}}w_{\delta,z})\Big|
\leq \frac{1}{\|w\|^2}\int_{\R^N}\Big|\frac{z}{1+|z|}-\frac{ x}{1+|x|}\Big|
\int_{\R^N}\frac{w^2_{\delta,z}(x) w^2_{\delta,z}(y)}{|x-y|^{4}}dy\,dx
\leq\delta,
$$
where we have used the fact that $|\frac{z}{1+|z|}-\frac{ x}{1+|x|}|<\delta$
for any $x\in B_{\delta}(z)$. Then
$$
\aligned
0
&\leq\gamma(\sqrt{k_{0}}w_{\delta,z},\sqrt{l_{0}}w_{\delta,z})\\
&=\frac{1}{\|w\|^{2}}\int_{\R^N}\Big|\frac{ x}{1+|x|}
-\beta(\sqrt{k_{0}}w_{\delta,z},\sqrt{l_{0}}w_{\delta,z})\Big|
\int_{\R^N}\frac{w^2_{\delta,z}(x) w^2_{\delta,z}(y)}{|x-y|^{4}}dy\,dx\\
&\leq\frac{1}{\|w\|^{2}}\int_{\R^N}\Big|\frac{ x}{1+|x|}-\frac{z}{1+|z|}\Big|
\int_{\R^N}\frac{w^2_{\delta,z}(x) w^2_{\delta,z}(y)}{|x-y|^{4}}dy\,dx\\
&\quad\,+\frac{1}{\|w\|^{2}}\int_{\R^N}\Big|\frac{z}{1+|z|}
-\beta(\sqrt{k_{0}}w_{\delta,z},\sqrt{l_{0}}w_{\delta,z})\Big|
\int_{\R^N}\frac{w^2_{\delta,z}(x) w^2_{\delta,z}(y)}{|x-y|^{4}}dy\,dx\\
&\leq2\delta,
\endaligned
$$
which implies that
$$
\lim_{\delta\rightarrow0^{+}}\gamma(\sqrt{k_{0}}w_{\delta,z},\sqrt{l_{0}}w_{\delta,z})=0
$$
uniformly for $z\in\R^N$.

(3) We first claim that
\begin{equation}\label{eq5.9}
\lim_{\delta\rightarrow+\infty}\beta(\sqrt{k_{0}}w_{\delta,z},\sqrt{l_{0}}w_{\delta,z})=0
\end{equation}
uniformly for $z\in B_{r}(0)$. Indeed, since $w_{\delta,0}$ is radially symmetric,
we have
$$
\int_{\R^N}\frac{ x}{1+|x|}
\int_{\R^N}\frac{w^2_{\delta,0}(x) w^2_{\delta,0}(y)}{|x-y|^{4}}dy\,dx=0
$$
and so
$$
\aligned
\Big|\beta(\sqrt{k_{0}}w_{\delta,z},\sqrt{l_{0}}w_{\delta,z})\Big|
&=\frac{1}{\|w\|^{2}}\Big|\int_{\R^N}\frac{ x}{1+|x|}
\int_{\R^N}\frac{w^2_{\delta,z}(x) w^2_{\delta,z}(y)}{|x-y|^{4}}dy\,)dx\Big|\\
&=\frac{1}{\|w\|^{2}}\Big|\int_{\R^N}\frac{ x}{1+|x|}
\int_{\R^N}\frac{w^2_{\delta,z}(x) w^2_{\delta,z}(y)
-w^2_{\delta,0}(x) w^2_{\delta,0}(y)}{|x-y|^{4}}dy\,dx\Big|\\
&\leq\frac{1}{\|w\|^{2}}\int_{\R^N}
\int_{\R^N}\frac{|w^2_{\delta,z}(x) w^2_{\delta,z}(y)
-w^2_{\delta,0}(x) w^2_{\delta,0}(y)|}{|x-y|^{4}}dy\,dx\\
&=\frac{1}{\|w\|^{2}}\int_{\R^N}
\int_{\R^N}\frac{|w^2_{1,z/\delta}(x) w^2_{1,z/\delta}(y)
-w^2_{1,0}(x) w^2_{1,0}(y)|}{|x-y|^{4}}dy\,dx\\
&\rightarrow0
\endaligned
$$
as $\delta\rightarrow+\infty$, uniformly for $z\in B_{r}(0)$.

For $\varepsilon> 0$, we fix a constant $\rho=\rho(\varepsilon)> 0$ such
that $\frac{1}{1+\rho}<\frac{\varepsilon}{3}$. For such a $\rho$, we see
from \eqref{eq5.9} that
$$
\lim_{\delta\rightarrow+\infty}\int_{B_{\rho}(0)}
\int_{\R^N}\frac{w^2_{\delta,z}(x) w^2_{\delta,z}(y)}{|x-y|^{4}}dy\,dx
=0
$$
uniformly for $z\in B_{r}(0)$ and that there exists $\delta_{0}> 0$ such that
$$
\big|\beta(\sqrt{k_{0}}w_{\delta,z},\sqrt{l_{0}}w_{\delta,z})\big|<\frac{\varepsilon}{3}
$$
and
$$
\frac{1}{\|w\|^{2}}\int_{B_{\rho}(0)}
\int_{\R^N}\frac{w^2_{\delta,z}(x) w^2_{\delta,z}(y)}{|x-y|^{4}}dy\,dx
<\frac{\varepsilon}{3}
$$
for all $\delta\in(\delta_{0},+\infty)$ and $z\in B_{r}(0)$. Observe that
$$
\gamma(\sqrt{k_{0}}w_{\delta,z},\sqrt{l_{0}}w_{\delta,z})
=\frac{1}{\|w\|^{2}}\int_{\R^N}\Big|\frac{ x}{1+|x|}
-\beta(\sqrt{k_{0}}w_{\delta,z},\sqrt{l_{0}}w_{\delta,z})\Big|
\int_{\R^N}\frac{w^2_{\delta,z}(x) w^2_{\delta,z}(y)}{|x-y|^{4}}dy\,dx
<1+\frac{\varepsilon}{3}
$$
for all $\delta\in(\delta_{0},+\infty)$ and $z\in B_{r}(0)$. On the other hand,
for all $\delta\in(\delta_{0},+\infty)$ and $z\in B_{r}(0)$ we have
$$
\aligned
\gamma(\sqrt{k_{0}}w_{\delta,z},\sqrt{l_{0}}w_{\delta,z})
&=\frac{1}{\|w\|^{2}}\int_{\R^N}\Big|\frac{ x}{1+|x|}
-\beta(\sqrt{k_{0}}w_{\delta,z},\sqrt{l_{0}}w_{\delta,z})\Big|
\int_{\R^N}\frac{w^2_{\delta,z}(x) w^2_{\delta,z}(y)}{|x-y|^{4}}dy\,dx\\
&\geq\frac{1}{\|w\|^{2}}\int_{\R^N}\frac{ |x|}{1+|x|}
\int_{\R^N}\frac{w^2_{\delta,z}(x) w^2_{\delta,z}(y)}{|x-y|^{4}}dy\,dx
-\frac{\varepsilon}{3}\\
&\geq\frac{1}{\|w\|^{2}}\int_{\R^N\backslash B_{\rho}(0)}\frac{ |x|}{1+|x|}
\int_{\R^N}\frac{w^2_{\delta,z}(x) w^2_{\delta,z}(y)}{|x-y|^{4}}dy\,dx-\frac{\varepsilon}{3}\\
&\geq\frac{\rho}{1+\rho}-\frac{1}{\|w\|^{2}}\int_{B_{\rho}(0)}
\int_{\R^N}\frac{w^2_{\delta,z}(x) w^2_{\delta,z}(y)}{|x-y|^{4}}dy\,dx-\frac{\varepsilon}{3}\\
&\geq1-\frac{1}{1+\rho}-\frac{\varepsilon}{3}-\frac{\varepsilon}{3}\\
&>1-\varepsilon.
\endaligned$$
Therefore, we have
$$
\lim_{\delta\rightarrow+\infty}\gamma(\sqrt{k_{0}}w_{\delta,z},\sqrt{l_{0}}w_{\delta,z})=1
$$ uniformly for $z\in B_{r}(0)$.
\end{proof}

For simplicity, we define $T:\R^+\times\R^N\rightarrow H$ by
$$
T(\delta,z)=(\sqrt{k_{0}}w_{\delta,z},\sqrt{l_{0}}w_{\delta,z})
$$
and $\Theta:H\backslash\{(0,0)\}\rightarrow\mathcal{N}$ by
$$
\Theta(u,v)=(t_{(u,v)}|u|,t_{(u,v)}|v|),
$$
where $t_{(u,v)} > 0$ is given by
$$
t_{(u,v)}^{2}=\frac{\|(u,v)\|^{2}+\int_{\R^N}( V_1(x)u^{2}+V_2(x)v^{2})dx}
{\displaystyle\int_{\R^N}\int_{\R^N}\frac{\alpha_1 u^2(x) u^2(y)+\alpha_2 v^2(x)v^2(y)
+2\beta u^2(x) v^2(y)}{|x-y|^{4}}dxdy}.
$$
We have, for  $\delta> 0$ and $z\in\R^N$,
$$
\aligned
J(\Theta\circ T(\delta,z))
&=\frac{1}{4}\frac{\displaystyle[(k_{0}+l_{0})\|w_{\delta,z}\|^{2}
+\int_{\R^N}(k_{0} V_1(x)w_{\delta,z}^{2}+ l_{0}V_2(x)w_{\delta,z}^{2})dx]^{2}}
{\displaystyle(k_{0}+l_{0})\int_{\R^N}\int_{\R^N}
\frac{w^2_{\delta,z}(x) w^2_{\delta,z}(y)}{|x-y|^{4}}dxdy}\\
&=\frac{1}{4}\frac{\displaystyle[(k_{0}+l_{0})\|w\|^{2}
+\int_{\R^N}(k_{0} V_1(x)w_{\delta,z}^{2}+ l_{0}V_2(x)w_{\delta,z}^{2})dx]^{2}}
{\displaystyle(k_{0}+l_{0})\int_{\R^N}\int_{\R^N}
\frac{w^2(x) w^2(y)}{|x-y|^{4}}dxdy}.
\endaligned$$
Then, by Lemma \ref{lem5.2} and by \eqref{eq5.4}, we can fix $R_0 > 0$ such that
$J(\Theta\circ T(\delta,z))<c^{\star}$ for all $\delta> 0$ and $|z|\geq R_0$.
Moreover, as a consequence of Lemmas \ref{lem5.2} and \ref{lem5.3}, we have

\begin{lem}\label{lem5.4}
$(1)$ There exists $\delta_{1}\in(0,\frac{1}{2})$ such that
$$
J(\Theta\circ T(\delta,z))<c^{\star}
$$
and
$$
\gamma(\sqrt{k_{0}}w_{\delta,z},\sqrt{l_{0}}w_{\delta,z})<\frac{1}{2}
$$
for all $\delta\in(0,\delta_{1}]$ and $z\in \R^N$.\\
$(2)$ There exists $\delta_{2}\in(\frac{1}{2},+\infty)$ such that
$$
J(\Theta\circ T(\delta,z))<c^{\star}
$$
and
$$
\gamma(\sqrt{k_{0}}w_{\delta,z},\sqrt{l_{0}}w_{\delta,z})>\frac{1}{2}
$$
for all $\delta\in[\delta_{2},+\infty)$ and $z\in B_{R_0}(0)$.
\end{lem}

\begin{lem}\label{lem5.5}
Denote $D=[\delta_{1},\delta_{2}]\times B_{R_0}(0)$ and define a map
$g:D\rightarrow\R^+\times\R^N$ by
$$
g(\delta,z)=(\gamma \circ\Theta\circ T(\delta,z),\beta\circ\Theta\circ T(\delta,z)).
$$
Then we have
$$
\text{\rm deg}\,\Big(g,D,\Big(\frac{1}{2},0\Big)\Big)=1.
$$
\end{lem}

\begin{proof}[\bf Proof.]
Consider the homotopy map $G : [0, 1]\times D\rightarrow \R^+\times\R^N$
defined by
$$
G(s,\delta,z)=(1-s)(\delta,z)+sg(\delta,z).
$$
We claim
\begin{equation}\label{eq5.10}
\Big(\frac{1}{2},0\Big)\not\in G([0, 1]\times \partial D).
\end{equation}
If this is true, then the conclusion follows easily from the homotopy invariance
and normalization of degree.

Now we verify \eqref{eq5.10}. If $\delta=\delta_{1}$ and $z\in B_{R_0}(0)$,
then we see from Lemma \ref{lem5.4}(1) that
$$
(1-s)\delta_{1}+s\gamma \circ\Theta\circ T(\delta_{1},z)
=(1-s)\delta_{1}+s\gamma(\sqrt{k_{0}}w_{\delta_1,z},\sqrt{l_{0}}w_{\delta_1,z})
<\frac{1}{2}.
$$
If $\delta=\delta_{2}$ and $z\in B_{R_0}(0)$, then it follows from Lemma
\ref{lem5.4}(2) that
$$
(1-s)\delta_{1}+s\gamma \circ\Theta\circ T(\delta_{2},z)
=(1-s)\delta_{2}+s\gamma(\sqrt{k_{0}}w_{\delta_2,z},\sqrt{l_{0}}w_{\delta_2,z})
>\frac{1}{2}.
$$
If $\delta\in[\delta_{1},\delta_{2}]$ and $|z|=R_0$, then using Lemma
\ref{lem5.3}(1) yields
$$
\langle(1-s)z+s\beta \circ\Theta\circ T(\delta,z),z\rangle
=\langle(1-s)z+s\beta(\sqrt{k_{0}}w_{\delta,z},\sqrt{l_{0}}w_{\delta,z}),z\rangle_{\R^N}
>0,
$$
which implies that $(1-s)z+s\beta \circ\Theta\circ T(\delta,z)\neq0$. Therefore,
\eqref{eq5.10} holds.
\end{proof}

Setting $\mathbb{A}=\Theta\circ T(D)$ and
$\Gamma=\{h\in C(\mathbb{A},\mathcal{N}):h|_{\partial\mathbb{A}}=id\}$,
we have

\begin{lem}\label{lem5.6}
$\mathcal{M}$ and $\partial\mathbb{A}$ link with respect to $\Gamma$.
\end{lem}

\begin{proof}[\bf Proof.]
Assume that $(u, v)\in\partial\mathbb{A}= \Theta\circ T(\partial D)$. From
the choice of $R_0$ and Lemma \ref{lem5.4}, we see that $J(u, v)<c^{\star}$
which implies $(u, v)\not\in\mathcal{M}$. Therefore, we have
$\mathcal{M}\cap \partial\mathbb{A}=\emptyset$.

For any $h\in \Gamma$, we define a continuous map
$\overline{g} : D\rightarrow \R^+\times\R^N$ by
$$
\overline{g}(\delta,z)
=(\gamma\circ h\circ\Theta\circ T(\delta,z),\beta\circ h\circ\Theta\circ T(\delta,z)).
$$
Since $h|_{\partial\mathbb{A}} = id$, we have
$\overline{g}|_{\partial D}=g|_{\partial D}$. Then it follows from Lemma
\ref{lem5.5} that
$$
\text{\rm deg}\,\Big(\overline{g},D,\Big(\frac{1}{2},0\Big)\Big)
=\text{\rm deg}\,\Big(g,D,\Big(\frac{1}{2},0\Big)\Big)=1.
$$
By the Kronecker existence theorem, there is
$(\overline{\delta},\overline{z})\in D$ such that
$h\circ\Theta\circ T(\overline{\delta},\overline{z})\in\mathcal{M}$. Then
we conclude that $\mathcal{M}\cap h(\mathbb{A})\neq\emptyset$.
\end{proof}

Now we are in a position to prove the main result.

\begin{proof}[\bf Proof of Theorem \ref{thm1.2}.]
Define the minimax value
$$
d=\inf_{h\in \Gamma}\max_{(u,v)\in\mathbb{A}}J(h(u,v)).
$$
Lemma \ref{lem5.6} indicates that
$\mathcal{M}\cap h(\mathbb{A})\neq\emptyset$ for any $h\in \Gamma$.
Then $d\geq c^{\star}>c_{\infty}$. Since $id\in\Gamma$, we have
$$
\aligned
d
&\leq \max_{(u,v)\in\mathbb{A}}J(u,v)\\
&\leq \max_{(\delta,z)\in\R^+\times\R^N}J(\Theta\circ T(\delta,z))\\
&\leq \max_{(\delta,z)\in\R^+\times\R^N}\frac{1}{4}
\frac{\displaystyle\big[(k_{0}+l_{0})\|w_{\delta,z}\|^{2}
+\int_{\R^N}(k_{0}V_1(x)w_{\delta,z}^{2}+l_{0}V_2(x)w_{\delta,z}^{2})dx\big]^{2}}
{\displaystyle(k_{0}+l_{0})\int_{\R^N}\int_{\R^N}
\frac{w^{2}_{\delta,z}(x) w^{2}_{\delta,z}(y)}{|x-y|^{4}}dxdy}.
\endaligned
$$
Using the H\"{o}lder inequality, the definition of $S$, \eqref{eq1.9} and \eqref{eq5.5}
leads to
\begin{equation}\label{e5.11}
\aligned
d
&\leq \frac{1}{4}\frac{\big[(k_{0}+l_{0})\|w\|^{2}
+k_{0}|V_1|_{\frac{N}{2}}|w|_{2^\ast}^{2}+l_{0}|V_2|_{\frac{N}{2}}|w|_{2^\ast}^{2}\big]^{2}}
{(k_{0}+l_{0})\|w\|^2}\\
&\leq\frac{1}{4}\frac{\big[(k_{0}+l_{0})\|w\|^{2}
+k_{0}S^{-1}|V_1|_{\frac{N}{2}}\|w\|^2+l_{0}S^{-1}|V_2|_{\frac{N}{2}}\|w\|^2\big]^{2}}
{(k_{0}+l_{0})\|w\|^2}\\
&=\frac{k_{0}+l_{0}}{4}\|w\|^{2}S_{H,L}^{-2}
\Big(S_{H,L}+\frac{\beta-\alpha_2}{2\beta-\alpha_1-\alpha_2}C(N,4)^{-\frac12}|V_1|_{\frac{N}{2}}
+\frac{\beta-\alpha_1}{2\beta-\alpha_1-\alpha_2}C(N,4)^{-\frac12}|V_2|_{\frac{N}{2}}\Big)^{2}\\
&<\min\Big\{\frac{S_{H,L}^{2}}{4\alpha_1}, \frac{S_{H,L}^{2}}{4\alpha_2}, 2c_{\infty}\Big\},
\endaligned
\end{equation}
where $2^\ast=\frac{2N}{N-2}$. According to the deformation lemma (see
\cite[Theorem 8.4 in Chapter II]{Sm2}), the constrained functional
$J|_\mathcal{N}$ has a Palais-Smale sequence $\{(u_n, v_n)\}\subset \cn$
at the level $d$. By Corollary \ref{cor4.3}, we conclude that $\{(u_n, v_n)\}$
contains a convergent subsequence in $H$, and there is a critical point $(u, v)$
of the constrained functional $J|_{\mathcal{N}}$ with $J(u, v) = d$. Then,
by Lemma \ref{lem2.1} and $d<\min\{S_{H,L}^{2}/4\alpha_1, S_{H,L}^{2}/4\alpha_2\}$,
it is easy to see that $u\neq0$, $v \neq 0$,
and $(u, v)$ is a critical point of the functional $J$. By the maximum principle,
we see that $(u, v)$ is a positive solution of \eqref{eq1.7}.
\end{proof}

\end{document}